\pdfoutput=1

\documentclass[12pt,reqno]{amsart}

\usepackage[margin=1in]{geometry}

\usepackage{amsmath,amssymb,amsfonts,amsthm}

\usepackage{mathptmx,cite,microtype,enumitem}

\usepackage{tikz}
\usepackage{pgfplots}
\pgfplotsset{compat=1.16}
\usepgfplotslibrary{fillbetween}
\usetikzlibrary{decorations.pathmorphing}

\usepackage[colorlinks=true,allcolors=blue]{hyperref}
\hypersetup{pdftitle={Near-unit-root persistence of symmetric and asymmetric stable autoregressive sequences},pdfauthor={J. Ricardo G. Mendonca, Boubaker Smii}}

\theoremstyle{plain}
\newtheorem{theorem}{Theorem}[section]
\newtheorem{proposition}[theorem]{Proposition}
\newtheorem{lemma}[theorem]{Lemma}
\newtheorem{corollary}[theorem]{Corollary}
\newtheorem{conjecture}[theorem]{Conjecture}
\theoremstyle{remark}
\newtheorem{remark}[theorem]{Remark}

\usepackage{etoolbox}
\AtEndEnvironment{remark}{\unskip\nobreak\hfill\ensuremath{\bigtriangleup}}

\renewcommand{\,}{\ifmmode\mkern2mu\else\thinspace\fi}

\newcommand{\NN}{\mathbb{N}}

\newcommand{\RR}{\mathbb{R}}

\newcommand{\PP}{\mathbb{P}}
\newcommand{\EE}{\mathbb{E}}

\newcommand{\deq}{\mathrel{\mathop:}=} 
\newcommand{\eqd}{=\mathrel{\mathop:}} 
\newcommand{\dis}{\mathrel{\stackrel{\mathrm d}{=}}} 

\renewcommand{\leq}{\leqslant}
\renewcommand{\geq}{\geqslant}

\newcommand{\1}{\text{\usefont{U}{bbold}{m}{n}1}}

\newcommand{\dd}{\,\mathrm{d}}
\newcommand{\SaS}{\mathrm{S}\alpha\mathrm{S}}

\title[Near-unit-root persistence of stable AR($1$) sequences]{Near-unit-root persistence of symmetric and asymmetric \\ stable autoregressive sequences}

\author[J. R. G. Mendon\c{c}a]{Jos\'{e} Ricardo G. Mendon\c{c}a$^{\star}$}
\address{Escola de Artes, Ci\^{e}ncias e Humanidades, Universidade de S\~{a}o Paulo, S\~{a}o Paulo, SP, Brazil}
\email{jricardo@usp.br}
\thanks{$^{\star}$Corresponding author. Email: {\href{mailto:jricardo@usp.br}{\nolinkurl{jricardo@usp.br}}}}

\author[B. Smii]{Boubaker Smii}
\address{Department of Mathematics, King Fahd University of Petroleum and Minerals, Dhahran, Saudi Arabia}
\email{boubaker@kfupm.edu.sa}

\date{\today}



\begin{document}

\begin{abstract}
For an AR($1$) sequence with coefficient $0 < a < 1$, the probability of staying above zero decays exponentially in the number of steps; at the unit root $a=1$ it decays polynomially, with order $n^{-1/2}$ for symmetric increments. We study this transition for AR($1$) sequences driven by symmetric and asymmetric $\alpha$-stable innovations. We write $\Lambda(a,\alpha)$ for the symmetric exponential persistence rate and $\Lambda(a,\alpha,\rho)$ for the rate of a strictly stable law with positivity parameter $\rho$. The entire chain admits an exact representation through a single stable L\'{e}vy process observed on a geometrically expanding time grid. In the symmetric case, comparison with continuous half-line survival gives $\Lambda(a,\alpha) \leq \frac{\alpha}{2}\log{(1/a)}$. For $0 < \alpha < 2$, this bound disproves the stable specialization of a conjecture of Hinrichs, Kolb and Wachtel concerning regularly varying innovation tails. Combining stable closure under subsampling with a monotonicity coupling yields a lower bound of the same near-unit order. This proves $\Lambda(a,\alpha) \asymp \log{(1/a)}$ as $a \uparrow 1$ and shows that $\Lambda(a,\alpha)/\log{(1/a)}$ converges to a limit in $(0,\alpha/2]$ equal to the supremum of that ratio over $0 < a < 1$. A Lamperti transformation reduces identification of this constant to a dense-sampling problem for a stationary stable Ornstein--Uhlenbeck process. Existing Gaussian theory determines the sharp value at $\alpha=2$; for $0 < \alpha < 2$, rescued crossings between observations remain the obstacle. For asymmetric strictly stable innovations with $\alpha \neq 1$ and admissible $\rho \in (0,1)$, the corresponding bounds are $\Lambda(a,\alpha,\rho) \leq \alpha(1-\rho)\log{(1/a)}$ and $\Lambda(a,\alpha,\rho) \asymp \log{(1/a)}$. At $\alpha=2$, only the symmetric case $\rho=1/2$ occurs.
\end{abstract}

\maketitle


\section{Introduction}
\label{sec:introduction}

\subsection{Model and persistence exponent}

Let $(\xi_{n})_{n \geq 1}$ be independent and identically distributed continuous, symmetric $\alpha$-stable random variables \cite{SamorodnitskyTaqqu1994}, with $0 < \alpha \leq 2$, normalized by
\begin{equation}
\label{eq:stable-cf}
\EE[e^{iq\xi_{1}}] = \exp(-\sigma^{\alpha} |q|^{\alpha}), \quad
q \in \RR, \quad \sigma > 0.
\end{equation}
Thus $\alpha=2$ corresponds to a centered Gaussian variable of variance $2\sigma^{2}$. We consider the AR($1$) recursion 
\begin{equation}
\label{eq:ar1}
X_{n} = aX_{n-1}+\xi_{n}, \quad 0 < a < 1, \quad X_{0} = x > 0,
\end{equation}
and its persistence probability
\begin{equation}
\label{eq:Qn}
Q_{n}(x;a,\alpha,\sigma) \,\deq\, \PP_{x}(X_{1}>0,\ldots,X_{n}>0) = \PP_{x}(T_{0}>n),
\end{equation}
where $T_{0} = \inf\{n \geq 1\colon X_{n} \leq 0\}$.

Hinrichs, Kolb and Wachtel \cite{HKW2020} developed the existence theory for the exponential rate in \eqref{eq:Qn}; we state the consequences used below.

\begin{proposition}[{Hinrichs--Kolb--Wachtel \cite{HKW2020}}]
\label{prop:hkw}
Symmetric $\alpha$-stable innovations satisfy the hypotheses of \cite[Theorem~1]{HKW2020}. Specifically, $\EE[\log{(1+|\xi_{1}|)}] < \infty$, $\EE[(\xi_{1}^{+})^{\delta}] < \infty$ for every $0 < \delta < \alpha$, and $\PP(\xi_{1} > 0)\PP(\xi_{1} < 0) > 0$. The following then hold:
\begin{enumerate}[label=\textup{(\roman*)}]
\item The limit
\begin{equation}
\label{eq:Lambda}
\Lambda(a,\alpha) \deq -\lim_{n \to \infty}\,\frac{1}{n}\log{Q_{n}(x;a,\alpha,\sigma)}
\end{equation}
exists in $(0,\infty)$ and is the same for every fixed starting point $x > 0$.
\item Writing $\pi$ for the stationary law of \eqref{eq:ar1},
\begin{equation}
\label{eq:stationary-start-rate}
-\lim_{n \to \infty}\,\frac{1}{n} \log{\PP_{\pi}(X_{0}>0, X_{1}>0,\ldots,X_{n}>0)} = \Lambda(a,\alpha).
\end{equation}
\end{enumerate}
\end{proposition}

Part (i) is \cite[Theorem~1]{HKW2020}. For part (ii), Hinrichs, Kolb and Wachtel introduce the modified stopping time $\widetilde{T}_{0} = \min\{k \geq 0\colon X_{k} \leq 0\}$. Their equation~(10), together with the identification $\lambda_{a}(\pi) = \lambda_{a}$ made just after their equation~(12), gives \eqref{eq:stationary-start-rate}. The condition $\pi[x,\infty) > 0$ that their argument requires holds for every $x > 0$ in the present setting by Corollary~\ref{cor:stationary-scale}. Moreover, dividing \eqref{eq:ar1} by $\sigma$ turns the chain into one with unit innovation scale started at $x/\sigma$, so
\[
Q_{n}(x;a,\alpha,\sigma) = Q_{n}(x/\sigma;a,\alpha,1),
\]
and the rate is also independent of $\sigma$.

The recursion \eqref{eq:ar1} is the exact skeleton of a stable Ornstein--Uhlenbeck process. Indeed, if
\begin{equation}
\label{eq:ou-sde}
\dd Y_{t} = -\theta Y_{t} \dd t + \eta \dd L_{t},
\end{equation}
where $L$ is a standard symmetric $\alpha$-stable L\'{e}vy process, then, for $j \geq 0$,
\begin{equation}
\label{eq:ou-skeleton}
Y_{(j+1)\Delta} = e^{-\theta\Delta}Y_{j\Delta} + \eta\int_{j\Delta}^{(j+1)\Delta} e^{-\theta((j+1)\Delta-s)} \dd L_{s}.
\end{equation}
The stochastic integrals over successive intervals are independent and identically distributed. For deterministic $f$, the integral $\int f \dd L$ is $\alpha$-stable with scale $\bigl(\int|f|^{\alpha}\bigr)^{1/\alpha}$; substituting $u=(j+1)\Delta-s$ in \eqref{eq:ou-skeleton} therefore makes the sampled process an AR($1$) sequence with
\begin{equation}
\label{eq:ou-innovation-scale}
a = e^{-\theta\Delta},
\quad
\sigma^{\alpha} =
\eta^{\alpha}\int_{0}^{\Delta}e^{-\alpha\theta u} \dd u =
\frac{\eta^{\alpha}(1-e^{-\alpha\theta\Delta})}{\alpha\theta};
\end{equation}
see, for example, \cite{Masuda2004}. Thus $a \uparrow 1$ is both a near-unit-root limit for the chain and a dense-sampling limit for the continuous-time process.

The stable AR($1$) chain can be represented by a single L\'{e}vy path sampled on a grid asymptotically uniform in logarithmic time. For $0 < \alpha < 2$, it remains open whether rescued sub-mesh crossings (excursions below zero completed between consecutive observations) alter the leading persistence rate.

\subsection{Main results}

Put
\begin{equation}
\label{eq:G-def}
G_{\alpha}(a) \deq \frac{\Lambda(a,\alpha)}{\log{(1/a)}},
\quad
G_{\alpha}^{*} \deq \sup_{0 < a < 1}G_{\alpha}(a).
\end{equation}

\begin{theorem}[Main theorem]
\label{thm:main}
For every $0 < \alpha \leq 2$ and $0 < a < 1$,
\begin{equation}
\label{eq:main-bounds}
0 < \Lambda(a,\alpha) \leq \min\Bigl\{\log{2},\frac{\alpha}{2}\log{(1/a)}\Bigr\}.
\end{equation}
For every fixed $a_{0} \in (0,1)$ and every $a \in [a_{0},1)$, one has the subsampling bound
\begin{equation}
\label{eq:ceiling-lower}
\Lambda(a,\alpha) \geq \frac{\Lambda(a_{0},\alpha)}{\bigl\lceil \log{(1/a_{0})}/\log{(1/a)} \bigr\rceil}.
\end{equation}
In particular,
\begin{equation}
\label{eq:quantitative-order}
\frac{\Lambda(a_{0},\alpha)}{\log{(1/a_{0})}+\log{(1/a)}}\log{(1/a)}
\leq \Lambda(a,\alpha) \leq \frac{\alpha}{2}\log{(1/a)}.
\end{equation}
Therefore,
\begin{equation}
\label{eq:near-unit-order}
\Lambda(a,\alpha) \asymp \log{(1/a)}, \quad a \uparrow 1,
\end{equation}
and
\begin{equation}
\label{eq:limit-sup-intro}
\lim_{a \uparrow 1}G_{\alpha}(a) = G_{\alpha}^{*} \in (0,\alpha/2].
\end{equation}
\end{theorem}

\begin{remark}[Scope of the main theorem]
For $0 < \alpha < 2$, the theorem leaves the value of $G_{\alpha}^{*}$ open; the Gaussian value at $\alpha=2$ follows from existing dense-sampling theory.
\end{remark}

The $\alpha/2$ upper bound in \eqref{eq:main-bounds} comes from an exact geometric-time representation,
\begin{equation}
\label{eq:embed-intro}
(X_{n})_{n \geq 0} \dis \bigl(a^{n}(x+L_{u_{n}})\bigr)_{n \geq 0},
\quad
u_{n} = \sigma^{\alpha}\sum_{j=1}^{n}a^{-\alpha j}.
\end{equation}
Continuous survival of the stable path implies survival on the grid $\{u_{n}\}$, and the one-sided stable exit exponent then gives the coefficient $\alpha/2$. For $0 < \alpha < 2$, this coefficient is strictly smaller than the value predicted by the regularly-varying-tail conjecture in \cite[Remark~20]{HKW2020}.

The lower estimate in \eqref{eq:quantitative-order} has a different origin. Subsampling every $m$-th observation produces another symmetric stable AR($1$) sequence, now with coefficient $a^{m}$. Dropping the intervening positivity constraints gives the subsampling inequality
\begin{equation}
\label{eq:refinement-intro}
\Lambda(a,\alpha) \geq \frac{1}{m}\Lambda(a^{m},\alpha),
\quad
G_{\alpha}(a) \geq G_{\alpha}(a^{m}).
\end{equation}
Combining it with monotonicity of $\Lambda$ in $a$ proves the near-unit lower bound and \eqref{eq:limit-sup-intro}.

A stationary formulation follows. If
\[
R_{t} = e^{-t/\alpha}L_{e^{t}}, \quad t \in \RR,
\]
is the Lamperti transform of stable L\'{e}vy motion and $\lambda_{\alpha}(h)$ denotes the persistence exponent of $R$ sampled at mesh $h$, then
\begin{equation}
\label{eq:lamperti-intro}
\Lambda(a,\alpha) = \lambda_{\alpha}(h),
\quad
h = \alpha\log{(1/a)},
\quad
G_{\alpha}^{*} = \alpha\lim_{h \downarrow 0}\frac{\lambda_{\alpha}(h)}{h} = \alpha\sup_{h>0}\frac{\lambda_{\alpha}(h)}{h}.
\end{equation}
The continuous-time persistence rate of $R$ is $1/2$. Thus the conjectural sharp value $G_{\alpha}^{*} = \alpha/2$ is equivalent to a dense-sampling assertion. In physical time, Corollary~\ref{cor:ou-dense} shows that the persistence rate of the $\Delta$-skeleton of \eqref{eq:ou-sde} converges to $\theta G_{\alpha}^{*}$. The sharp conjecture is that this limit equals $\alpha\theta/2$, the continuously monitored rate.

In the proof of the near-unit bounds, symmetry enters essentially through the one-sided exit exponent and the lower bound on each marginal $\PP(X_{k}>0)$. Section~\ref{sec:bounds} treats both ingredients using the positivity parameter and obtains, for asymmetric strictly stable innovations with $\alpha\neq1$ and admissible $\rho = \PP(\xi_{1}>0) \in (0,1)$,
\begin{equation}
\label{eq:asym-ceiling-intro}
\Lambda(a,\alpha,\rho) \leq \alpha(1-\rho)\log{(1/a)},
\end{equation}
together with the corresponding near-unit limit. Figure~\ref{fig:bounds} summarizes both sets of bounds.


\begin{figure}[t]
\centering
\hspace*{-1.25cm}%
\begin{tikzpicture}
\begin{axis}[name=pa, width=7.2cm, height=5.5cm, xlabel={$a$}, ylabel={$G_{\alpha}(a)$}, xmin=0.08, xmax=1.0, ymin=0, ymax=0.9, xtick={0.25,0.50,0.75,1.0}, ytick={0,0.25,0.5,0.75}, tick label style={font=\scriptsize}, label style={font=\small}, every axis plot/.append style={line width=0.75pt}]
\draw[black!50] (axis cs:0.3969,0) -- (axis cs:0.3969,0.75);
\node[font=\scriptsize,anchor=south west] at (axis cs:0.4029,0.01) {$a^{*}$};
\addplot[no marks,black,domain=0.08:0.3969,samples=60] {ln(2)/ln(1/x)};
\addplot[no marks,black] coordinates {(0.3969,0.75000) (0.9990,0.75000)};
\addplot[no marks,black,densely dotted,domain=0.5:0.999,samples=120] {0.55*ln(2)/(ceil(ln(2)/ln(1/x))*ln(1/x))};
\addplot[no marks,black!50,densely dashed,domain=0.5:0.999,samples=30] {0.55*ln(2)/ln(2/x)};
\draw[black!50,densely dashed,thick] (axis cs:0.405,0.55) -- (axis cs:1.0,0.55);
\node[font=\scriptsize,anchor=south] at (axis cs:0.62,0.750) {$\alpha/2$};
\node[font=\scriptsize,anchor=south west] at (axis cs:0.42,0.555) {$G_{\alpha}(a_{0})$};
\node[font=\small,anchor=north west,inner sep=2pt] at (rel axis cs:0.02,0.98) {(a)};
\end{axis}
\begin{axis}[name=pb, at={(pa.right of south east)}, xshift=2.0cm, width=7.2cm, height=5.5cm, xlabel={$\alpha$}, ylabel={$\alpha(1-\rho)$}, xmin=0, xmax=2.025, ymin=0, ymax=1.20, xtick={0,0.5,1,1.5,2}, ytick={0,0.25,0.5,0.75,1}, tick label style={font=\scriptsize}, label style={font=\small}, every axis plot/.append style={line width=0.75pt}]
\addplot[no marks,black,name path=T] coordinates {(0.000,0.000) (1.000,1.000) (2.000,1.000)};
\addplot[no marks,black,name path=B] coordinates {(0.000,0.003) (1.000,0.003) (2.000,1.000)};
\addplot[black!7] fill between [of=T and B];
\addplot[no marks,black,densely dashed] coordinates {(0.010,0.005) (2.000,1.000)};
\node[font=\scriptsize,anchor=south] at (axis cs:1.50,1.00) {$\beta=1$};
\node[font=\scriptsize,anchor=north] at (axis cs:1.50,0.35) {$\beta=-1$};
\node[font=\scriptsize,fill=black!7,inner sep=2pt] at (axis cs:1.020,0.50) {$\rho=\tfrac{1}{2}$};
\node[font=\small,anchor=north west,inner sep=2pt] at (rel axis cs:0.02,0.98) {(b)};
\fill (axis cs:1.995,1) circle (2.25pt);
\end{axis}
\end{tikzpicture}%
\hspace*{1.25cm}
\caption{Persistence-rate bounds. (a) Schematic symmetric case, drawn for $\alpha=3/2$. The solid curve is the proved ceiling in \eqref{eq:main-bounds}, whose two branches meet at $a^{*}=2^{-2/\alpha}$; the dotted staircase is \eqref{eq:ceiling-lower} anchored at $a_{0}=1/2$, and the dashed curve is \eqref{eq:quantitative-order}. The unknown value $G_{\alpha}(a_{0})$ is drawn at an arbitrary admissible height. (b) The closure of the admissible ceilings $\alpha(1-\rho)$ is the shaded region between $\max\{0,\alpha-1\}$ and $\min\{\alpha,1\}$. The dashed line is the symmetric value $\alpha/2$. The labels $\beta=\pm1$ refer to $1 < \alpha < 2$; for $0<\alpha<1$ the two boundaries exchange roles, $\beta=-1$ giving the upper boundary $\alpha$ and $\beta=1$ the lower boundary $0$. The slice $\alpha=1$ is conditional; see Remark~\ref{rem:cauchy-gap}. At $\alpha=2$ the region reduces to the Gaussian point $(2,1)$.}
\label{fig:bounds}
\end{figure}
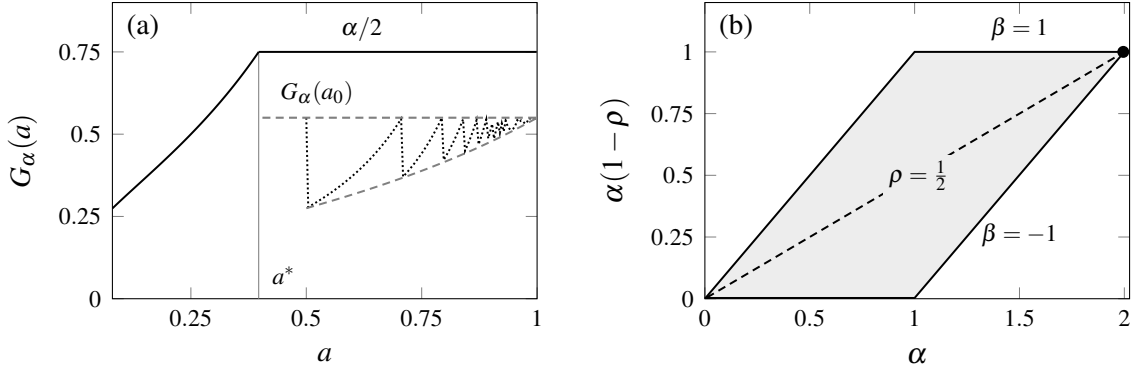

\subsection{Relation to previous work}

For surveys of persistence and first-passage problems in probability and statistical physics, see \cite{BrayMajumdarSchehr2013,AurzadaSimon2015}. At $a=1$, \eqref{eq:ar1} is a random walk. Extending \eqref{eq:Qn} to that case and to the boundary start $x=0$, the Sparre Andersen theorem gives, for continuous symmetric increments \cite{SparreAndersen1953},
\begin{equation}
\label{eq:sparre-andersen}
Q_{n}(0;1,\alpha,\sigma) = 2^{-2n}\binom{2n}{n} \sim \frac{1}{\sqrt{\pi n}}.
\end{equation}
For a fixed positive start the exponent remains $1/2$, while the amplitude depends on the descending-ladder renewal function.

For Gaussian innovations, Aurzada and Baumgarten \cite[Section~2.3]{AurzadaBaumgarten2011} related geometric sampling to a stationary Gaussian Ornstein--Uhlenbeck process and proved $\lambda_{\beta} \leq \beta/2$ for small $\beta$ together with coarse-grid lower bounds; their argument also gives the linear near-unit order $\lambda_{\beta} \asymp \beta$. The sharp Gaussian dense-sampling limit follows from Feldheim, Feldheim and Mukherjee \cite[Theorem~5(II) and Remark~6]{FeldheimFM2025}. Majumdar, Bray and Ehrhardt \cite{MajumdarBrayEhrhardt2001} derived the half-line integral eigenvalue problem governing persistence of the discretely sampled Gaussian Ornstein--Uhlenbeck process and studied the discrete-sampling correction.

For non-Gaussian AR($1$) sequences, Hinrichs, Kolb and Wachtel \cite{HKW2020} provide the existence theory used here; they also write $a$ for the autoregressive coefficient, and their additive decay rate $\lambda_{a}$ is our $\Lambda(a,\alpha)$. For continuous symmetric innovations, Alsmeyer et al.~\cite{AlsmeyerBRS2023} obtained a Baxter--Spitzer factorization. In our coefficient notation, the multiplicative persistence rate studied by Donnart and Simon \cite{DonnartSimon2026} is $e^{-\Lambda(a,\alpha)}$. They identify this multiplicative persistence rate as $\mu_{a}^{-1}$, where $\mu_{a}$ denotes their threshold defined through the dual first-passage generating function. Under their nonessential-singularity condition, $\mu_{a}$ is the first positive root of the resulting generating-function equation. This implicit characterization neither evaluates $\mu_{a}$ for stable innovations nor determines its behavior as $a \uparrow 1$. For comparison with the decay exponent, Donnart and Simon \cite[Section~4.2]{DonnartSimon2026} show in their symmetric bi-exponential example that the mean first-passage time $\EE_{x}[T_{0}]$ is of order $(1-a)^{-1/2}$ as $a \uparrow 1$. Related directions include logarithmic tails \cite{DenisovHKW2022}, discrete innovations \cite{VysotskyWachtel2023}, parameter continuity \cite{AurzadaMukherjeeZeitouni2021}, and first-passage questions for certain stable Ornstein--Uhlenbeck processes \cite{Patie2008}.


\section{The geometric-time representation}
\label{sec:embedding}

Let $L = (L_{t})_{t \geq 0}$ be the standard symmetric $\alpha$-stable L\'{e}vy process,
\begin{equation}
\label{eq:levy-cf}
\EE[e^{iqL_{t}}] = e^{-t|q|^{\alpha}}.
\end{equation}
It is $1/\alpha$-self-similar:
\[
(L_{ct})_{t \geq 0} \,\dis\, (c^{1/\alpha}L_{t})_{t \geq 0}, \quad c > 0.
\]

\begin{theorem}[Geometric-time embedding]
\label{thm:embedding}
Let $X$ satisfy \eqref{eq:ar1} with innovations normalized by \eqref{eq:stable-cf}, and set
\begin{equation}
\label{eq:un}
u_{0} = 0, \quad u_{n} \,\deq\, \sigma^{\alpha}\sum_{j=1}^{n}a^{-\alpha j} =
\sigma^{\alpha}\biggl(\frac{a^{-\alpha n}-1}{1-a^{\alpha}}\biggr).
\end{equation}
Then, as random elements of $\RR^{\NN_{0}}$,
\begin{equation}
\label{eq:embedding}
(X_{n})_{n \geq 0} \,\dis\, (a^{n}(x+L_{u_{n}}))_{n \geq 0}.
\end{equation}
In particular,
\begin{equation}
\label{eq:sampled-persistence}
Q_{n}(x;a,\alpha,\sigma) = \PP(x+L_{u_{k}}>0,~ 1 \leq k \leq n).
\end{equation}
\end{theorem}

\begin{proof}
Iterating \eqref{eq:ar1} gives
\[
X_{n} = a^{n} x+\sum_{j=1}^{n}a^{n-j}\xi_{j} =
a^{n}\biggl(x+\sum_{j=1}^{n}a^{-j}\xi_{j}\biggr).
\]
Let $S_{n} = \sum_{j=1}^{n}a^{-j}\xi_{j}$. Its increments are independent and
\[
S_{n}-S_{n-1} = a^{-n}\xi_{n}
\]
is symmetric $\alpha$-stable with scale $\sigma a^{-n}$. On the other hand,
\[
u_{n}-u_{n-1} = \sigma^{\alpha}a^{-\alpha n},
\]
and \eqref{eq:levy-cf} gives $L_{t}$ the stable scale $t^{1/\alpha}$, so $L_{u_{n}}-L_{u_{n-1}}$ has scale $(\sigma^{\alpha}a^{-\alpha n})^{1/\alpha} = \sigma a^{-n}$, the same as $S_{n}-S_{n-1}$. The two sequences have independent increments with identical increment laws. Hence $(S_{n})_{n \geq 0} \dis (L_{u_{n}})_{n \geq 0}$ on the path space $\RR^{\NN_{0}}$. This proves \eqref{eq:embedding}. Since $a^{k} > 0$, the persistence identity follows.
\end{proof}

\begin{corollary}[Marginal and stationary scales]
\label{cor:stationary-scale}
For every $n \geq 1$, the centered random part of $X_{n}$ is symmetric $\alpha$-stable with scale
\begin{equation}
\label{eq:marginal-scale}
\sigma_{n} = \sigma\biggl(\frac{1-a^{\alpha n}}{1-a^{\alpha}}\biggr)^{1/\alpha}.
\end{equation}
The unique stationary law is symmetric $\alpha$-stable with scale
\begin{equation}
\label{eq:stationary-scale}
\sigma_{\infty} = \sigma(1-a^{\alpha})^{-1/\alpha}.
\end{equation}
In particular, the stationary law has support all of $\RR$.
\end{corollary}

\begin{proof}
The random sum
\[
\sum_{j=1}^{n}a^{n-j}\xi_{j}
\]
has stable scale
\[
\sigma\biggl(\sum_{j=1}^{n}a^{\alpha(n-j)}\biggr)^{1/\alpha} =\, \sigma\biggl(\sum_{k=0}^{n-1}a^{\alpha k}\biggr)^{1/\alpha},
\]
and the geometric sum $\sum_{k=0}^{n-1}a^{\alpha k} = (1-a^{\alpha n})/(1-a^{\alpha})$ gives \eqref{eq:marginal-scale}. Letting $n \to \infty$ gives \eqref{eq:stationary-scale}. To verify that this law is stationary, note that stability reduces invariance under $z \mapsto az+\xi_{1}$ to the scale identity
\[
\sigma_{\infty}^{\alpha} = a^{\alpha}\sigma_{\infty}^{\alpha}+\sigma^{\alpha},
\]
which is satisfied by \eqref{eq:stationary-scale}. To prove uniqueness, start the recursion from any stationary law, independently of the innovations, and iterate it. The initial term $a^{n}X_{0}$ converges almost surely to zero, while the innovation sum converges in distribution to the stable law in \eqref{eq:stationary-scale}. Stationarity forces the initial law to be that limit.
\end{proof}

The grid is asymptotically uniform in logarithmic time. With $\delta = \log{(1/a)}$,
\begin{equation}
\label{eq:grid-log}
u_{n} = \frac{\sigma^{\alpha}}{1-a^{\alpha}} (e^{\alpha\delta n}-1), \quad
\log{u_{n}} = \alpha\delta n+O(1) \quad (n \to \infty).
\end{equation}
Moreover,
\begin{equation}
\label{eq:grid-ratio}
\frac{u_{n}-u_{n-1}}{u_{n}} \longrightarrow 1-a^{\alpha}, \quad
\log{\frac{u_{n+1}}{u_{n}}} \longrightarrow \alpha\delta.
\end{equation}
Thus the spacing $u_{n}-u_{n-1}$ diverges in the time scale of $L$, whereas the logarithmic spacing $\log{(u_{n+1}/u_{n})}$ converges to $\alpha\delta$, with $\alpha\delta \to 0$ as $a \uparrow 1$. Figure~\ref{fig:grid} shows both facts on one simulated path, together with a crossing between two observations that the chain does not record.

\begin{remark}[Nonzero thresholds]
\label{rem:threshold}
For a fixed level $b$,
\[
X_{k} > b \quad\Longleftrightarrow\quad x+L_{u_{k}} > b a^{-k}.
\]
Hence persistence above $b$ becomes a stable path problem with a geometrically moving boundary. For $b=0$, the moving boundary reduces to the constant boundary $0$.
\end{remark}

\begin{figure}[t]
\centering
\begin{tikzpicture}
\begin{axis}[name=pa, width=13.2cm, height=5.0cm,
  xmin=0, xmax=59.117, ymin=-8.9, ymax=9.0,
  xlabel={$t$ (time scale of $L$)}, ylabel={$x+L_{t}$},
  xtick={1.707,4.623,9.601,18.101,32.614,57.395}, xticklabels={$u_{1}$,$u_{2}$,$u_{3}$,$u_{4}$,$u_{5}$,$u_{6}$},
  ytick={-9,-6,-3,0,3,6,9},
  tick label style={font=\scriptsize}, label style={font=\small},
  every axis plot/.append style={line width=0.6pt},
  axis on top]
\fill[black!7] (axis cs:0,-8.9) rectangle (axis cs:59.117,0);
\draw[black!50,densely dashed] (axis cs:1.707,\pgfkeysvalueof{/pgfplots/ymin}) -- (axis cs:1.707,\pgfkeysvalueof{/pgfplots/ymax});
\draw[black!50,densely dashed] (axis cs:4.623,\pgfkeysvalueof{/pgfplots/ymin}) -- (axis cs:4.623,\pgfkeysvalueof{/pgfplots/ymax});
\draw[black!50,densely dashed] (axis cs:9.601,\pgfkeysvalueof{/pgfplots/ymin}) -- (axis cs:9.601,\pgfkeysvalueof{/pgfplots/ymax});
\draw[black!50,densely dashed] (axis cs:18.101,\pgfkeysvalueof{/pgfplots/ymin}) -- (axis cs:18.101,\pgfkeysvalueof{/pgfplots/ymax});
\draw[black!50,densely dashed] (axis cs:32.614,\pgfkeysvalueof{/pgfplots/ymin}) -- (axis cs:32.614,\pgfkeysvalueof{/pgfplots/ymax});
\draw[black!50,densely dashed] (axis cs:57.395,\pgfkeysvalueof{/pgfplots/ymin}) -- (axis cs:57.395,\pgfkeysvalueof{/pgfplots/ymax});
\addplot[no marks,black] coordinates {(0.135,1.489) (0.137,1.494) (0.138,1.446) (0.139,1.447) (0.141,1.461) (0.142,1.469) (0.144,1.485) (0.145,1.496) (0.147,1.489) (0.148,1.475) (0.150,1.485) (0.151,1.481) (0.153,1.482) (0.154,1.490) (0.156,1.458) (0.157,1.412) (0.159,1.408) (0.160,1.417) (0.162,1.430) (0.164,1.425) (0.165,1.408) (0.167,1.443) (0.169,1.437) (0.170,1.448) (0.172,1.421) (0.174,1.395) (0.176,1.451) (0.177,1.473) (0.179,1.471) (0.181,1.502) (0.183,1.488) (0.185,1.511) (0.186,1.548) (0.188,1.547) (0.190,1.560) (0.192,1.580) (0.194,1.580) (0.196,1.562) (0.198,1.471) (0.200,1.347) (0.202,1.365) (0.204,1.376) (0.206,1.380) (0.208,1.386) (0.210,1.394) (0.212,1.355) (0.214,1.358) (0.217,1.363) (0.219,1.363) (0.221,1.395) (0.223,1.349) (0.225,1.365) (0.228,1.377) (0.230,1.349) (0.232,1.384) (0.235,1.331) (0.237,1.353) (0.239,1.365) (0.242,1.356) (0.244,1.398) (0.247,1.384) (0.249,1.397) (0.252,1.409) (0.254,1.435) (0.257,1.426) (0.259,1.483) (0.262,1.469) (0.264,1.447) (0.267,1.485) (0.270,1.465) (0.273,1.533) (0.275,1.532) (0.278,1.520) (0.281,1.505) (0.284,1.503) (0.287,1.392) (0.289,1.372) (0.292,1.382) (0.295,1.394) (0.298,1.364) (0.301,1.373) (0.304,1.356) (0.307,1.350) (0.310,1.364) (0.313,1.430) (0.317,1.424) (0.320,1.428) (0.323,1.444) (0.326,1.425) (0.330,0.933) (0.333,0.811) (0.336,0.828) (0.340,0.848) (0.343,0.793) (0.346,0.815) (0.350,0.934) (0.353,0.921) (0.357,0.964) (0.361,1.003) (0.364,0.931) (0.368,0.923) (0.372,0.895) (0.375,0.900) (0.379,0.881) (0.383,0.916) (0.387,0.858) (0.391,0.869) (0.395,0.850) (0.399,0.868) (0.403,0.843) (0.407,0.900) (0.411,0.903) (0.415,0.887) (0.419,1.023) (0.423,1.021) (0.427,0.984) (0.432,0.922) (0.436,0.965) (0.440,0.957) (0.445,0.973) (0.449,1.008) (0.454,1.002) (0.458,1.013) (0.463,1.022) (0.468,1.034) (0.472,1.028) (0.477,1.017) (0.482,0.939) (0.487,0.971) (0.492,1.016) (0.497,1.317) (0.502,0.866) (0.507,0.852) (0.512,0.817) (0.517,0.799) (0.522,0.761) (0.527,0.832) (0.533,0.814) (0.538,0.806) (0.543,0.831) (0.549,0.898) (0.554,0.905) (0.560,0.918) (0.566,0.893) (0.571,0.864) (0.577,0.853) (0.583,0.790) (0.589,0.724) (0.595,0.648) (0.600,0.589) (0.607,0.681) (0.613,0.636) (0.619,0.689) (0.625,0.684) (0.631,0.623) (0.638,0.662) (0.644,0.643) (0.651,0.701) (0.657,0.783) (0.664,0.832) (0.670,0.815) (0.677,0.859) (0.684,0.880) (0.691,0.843) (0.698,0.815) (0.705,0.791) (0.712,0.833) (0.719,0.857) (0.726,0.848) (0.733,1.052) (0.741,1.111) (0.748,1.137) (0.756,1.125) (0.763,1.104) (0.771,1.143) (0.779,1.191) (0.787,1.120) (0.795,1.154) (0.803,1.077) (0.811,1.089) (0.819,1.130) (0.827,1.075) (0.835,1.020) (0.844,1.065) (0.852,1.165) (0.861,1.160) (0.869,1.200) (0.878,0.820) (0.887,0.816) (0.896,0.874) (0.905,0.801) (0.914,0.846) (0.923,0.941) (0.932,0.870) (0.942,0.894) (0.951,0.870) (0.961,0.953) (0.970,0.962) (0.980,0.908) (0.990,0.927) (1.000,0.917) (1.010,0.962) (1.020,0.939) (1.030,0.911) (1.041,1.074) (1.051,0.991) (1.062,1.052) (1.073,0.965) (1.083,0.966) (1.094,0.936) (1.105,0.964) (1.116,0.993) (1.127,0.967) (1.139,1.032) (1.150,1.086) (1.162,1.027) (1.174,0.934) (1.185,1.247) (1.197,1.179) (1.209,1.106) (1.221,1.017) (1.234,1.089) (1.246,1.064) (1.259,1.052) (1.271,0.989) (1.284,0.966) (1.297,1.325) (1.310,1.141) (1.323,1.146) (1.336,0.943) (1.350,0.999) (1.363,0.988) (1.377,1.053) (1.391,1.063) (1.405,1.159) (1.419,1.186) (1.433,1.299) (1.448,1.414) (1.462,1.661) (1.477,1.735) (1.492,1.691) (1.507,1.615) (1.522,1.445) (1.537,1.389) (1.553,1.497) (1.568,1.576) (1.584,1.534) (1.600,1.494) (1.616,1.483) (1.632,1.436) (1.649,1.460) (1.665,1.334) (1.682,1.404) (1.699,1.677) (1.716,1.679) (1.733,1.498) (1.751,1.522) (1.768,1.598) (1.786,1.699) (1.804,1.849) (1.822,1.902) (1.840,2.081) (1.859,2.166) (1.878,2.097) (1.896,2.202) (1.916,2.219) (1.935,2.175) (1.954,2.045) (1.974,2.147) (1.994,2.093) (2.014,2.123) (2.034,2.114) (2.054,2.218) (2.075,2.499) (2.096,2.804) (2.117,2.901) (2.138,2.828) (2.160,2.754) (2.181,2.659) (2.203,2.585) (2.226,2.704) (2.248,2.805) (2.270,2.910) (2.293,3.058) (2.316,3.013) (2.340,2.858) (2.363,2.987) (2.387,3.114) (2.411,3.046) (2.435,3.141) (2.460,3.045) (2.484,2.868) (2.509,2.905) (2.535,2.872) (2.560,2.700) (2.586,2.724) (2.612,2.665) (2.638,1.746) (2.664,1.713) (2.691,1.624) (2.718,1.754) (2.746,1.670) (2.773,1.624) (2.801,1.602) (2.829,1.611) (2.858,1.449) (2.886,1.580) (2.915,1.596) (2.945,1.671) (2.974,1.405) (3.004,1.269) (3.034,1.153) (3.065,1.021) (3.096,1.032) (3.127,1.084) (3.158,1.300) (3.190,1.372) (3.222,1.200) (3.254,0.664) (3.287,0.583) (3.320,0.582) (3.353,0.712) (3.387,0.775) (3.421,0.772) (3.456,1.234) (3.490,1.355) (3.525,1.281) (3.561,1.160) (3.597,1.130) (3.633,1.050) (3.669,1.145) (3.706,2.344) (3.743,2.407) (3.781,2.617) (3.819,2.450) (3.857,2.350) (3.896,2.513) (3.935,2.752) (3.975,2.791) (4.015,2.452) (4.055,2.231) (4.096,2.126) (4.137,2.257) (4.179,2.215) (4.221,2.098) (4.263,2.183) (4.306,2.236) (4.349,1.770) (4.393,1.876) (4.437,1.961) (4.482,2.072) (4.527,3.287) (4.572,3.187) (4.618,3.647) (4.665,3.642) (4.711,3.623) (4.759,3.648) (4.807,3.482) (4.855,2.777) (4.904,2.945) (4.953,3.070) (5.003,2.939) (5.053,2.842) (5.104,3.144) (5.155,2.961) (5.207,2.765) (5.259,2.588) (5.312,2.622) (5.366,2.779) (5.419,2.434) (5.474,2.241) (5.529,2.182) (5.585,2.393) (5.641,2.338) (5.697,2.524) (5.755,2.480) (5.812,2.597) (5.871,2.334) (5.930,2.333) (5.989,2.400) (6.050,2.568) (6.110,2.793) (6.172,2.606) (6.234,2.734) (6.297,2.593) (6.360,2.470) (6.424,2.738) (6.488,2.507) (6.554,2.654) (6.619,2.518) (6.686,2.544) (6.753,2.654) (6.821,2.953) (6.890,3.074) (6.959,2.778) (7.029,2.754) (7.099,2.566) (7.171,2.465) (7.243,2.496) (7.316,2.790) (7.389,2.951) (7.463,2.251) (7.538,2.570) (7.614,3.378) (7.691,3.367) (7.768,3.380) (7.846,2.976) (7.925,2.764) (8.004,2.969) (8.085,2.732) (8.166,2.541) (8.248,2.287) (8.331,2.419) (8.415,2.406) (8.499,2.259) (8.585,2.079) (8.671,1.931) (8.758,1.574) (8.846,1.959) (8.935,1.503) (9.025,1.676) (9.116,1.854) (9.207,1.882) (9.300,1.691) (9.393,1.501) (9.488,1.458) (9.583,1.007) (9.679,1.039) (9.777,1.043) (9.875,1.117) (9.974,0.830) (10.074,0.604) (10.176,0.425) (10.278,0.755) (10.381,0.551) (10.486,0.728) (10.591,-0.253) (10.697,-0.327) (10.805,-0.447) (10.913,-0.305) (11.023,-0.489) (11.134,-0.452) (11.246,-0.513) (11.359,-0.326) (11.473,-0.410) (11.588,-0.063) (11.705,-0.408) (11.822,4.318) (11.941,4.270) (12.061,4.175) (12.182,4.189) (12.305,3.946) (12.429,3.929) (12.554,4.077) (12.680,3.528) (12.807,3.283) (12.936,3.341) (13.066,3.617) (13.197,3.604) (13.330,3.993) (13.464,4.715) (13.599,4.906) (13.736,5.236) (13.874,6.353) (14.013,5.607) (14.154,5.720) (14.296,5.727) (14.440,6.407) (14.585,5.661) (14.732,5.461) (14.880,5.761) (15.029,6.314) (15.180,4.057) (15.333,4.235) (15.487,5.100) (15.643,5.076) (15.800,4.943) (15.959,4.815) (16.119,6.668) (16.281,6.793) (16.445,6.706) (16.610,7.060) (16.777,7.016) (16.945,7.033) (17.116,7.327) (17.288,7.262) (17.462,7.196) (17.637,7.028) (17.814,7.027) (17.993,6.816) (18.174,6.795) (18.357,6.709) (18.541,6.082) (18.728,6.157) (18.916,6.387) (19.106,6.907) (19.298,7.100) (19.492,7.083) (19.688,7.677) (19.886,6.850) (20.086,7.391) (20.287,6.993) (20.491,7.470) (20.697,8.063) (20.905,7.848) (21.115,7.691) (21.328,2.995) (21.542,1.798) (21.758,1.258) (21.977,1.581) (22.198,1.938) (22.421,2.068) (22.646,3.020) (22.874,3.398) (23.104,4.007) (23.336,3.378) (23.571,4.039) (23.807,4.222) (24.047,4.486) (24.288,5.259) (24.533,5.840) (24.779,5.112) (25.028,5.261) (25.280,5.137) (25.534,5.497) (25.790,5.284) (26.050,6.122) (26.311,5.628) (26.576,6.216) (26.843,6.039) (27.113,5.421) (27.385,5.188) (27.660,4.672) (27.938,5.706) (28.219,5.138) (28.503,4.941) (28.789,4.160) (29.079,3.189) (29.371,2.292) (29.666,0.005) (29.964,0.021) (30.265,-0.860) (30.569,-0.880) (30.877,-0.572) (31.187,-0.952) (31.500,-1.097) (31.817,-1.270) (32.137,-1.583) (32.460,-2.779) (32.786,-4.067) (33.115,-4.995) (33.448,-4.489) (33.784,-4.255) (34.124,-4.443) (34.467,-5.006) (34.813,-3.798) (35.163,-3.928) (35.517,-4.384) (35.874,-5.053) (36.234,-5.106) (36.598,-7.397) (36.966,-7.592) (37.338,-8.304) (37.713,-7.760) (38.092,-7.889) (38.475,-7.715) (38.861,-8.032) (39.252,-7.387) (39.646,-7.317) (40.045,-6.710) (40.447,-5.108) (40.854,-2.842) (41.264,-2.927) (41.679,-2.660) (42.098,-2.926) (42.521,-1.831) (42.948,-0.179) (43.380,-0.534) (43.816,0.944) (44.256,1.370) (44.701,1.588) (45.150,0.704) (45.604,-2.225) (46.063,-3.031) (46.525,-0.914) (46.993,-1.352) (47.465,-1.948) (47.942,-4.430) (48.424,-4.683) (48.911,-4.960) (49.402,-4.422) (49.899,-4.937) (50.400,-4.716) (50.907,-5.095) (51.419,-3.774) (51.935,-3.818) (52.457,-2.986) (52.985,-3.425) (53.517,-3.533) (54.055,-4.494) (54.598,-5.144) (55.147,-6.510) (55.701,-6.469) (56.261,-5.364) (56.826,-5.038) (57.397,-5.357) (57.974,-3.222) (58.557,-0.499)};
\addplot[only marks,mark=*,mark size=1.9pt,black] coordinates {(1.707,1.758) (4.623,3.628) (9.601,1.004) (18.101,6.699)};
\addplot[only marks,mark=o,mark size=2.1pt,black,line width=0.8pt] coordinates {(32.614,-3.262) (57.395,-5.357)};
\draw[->,>=stealth,black!70] (axis cs:14.4,-4.2) -- (axis cs:11.7,-1.0);
\node[font=\scriptsize,anchor=west,black!70,fill=black!7,inner sep=1pt] at (axis cs:13.8,-4.6) {crossing between $u_{3}$ and $u_{4}$};
\node[font=\small,anchor=north west,inner sep=2pt] at (rel axis cs:0.02,0.98) {(a)};
\end{axis}
\begin{axis}[name=pb, at={(pa.below south west)}, anchor=above north west, yshift=-0.4cm,
  width=13.2cm, height=5.0cm,
  xmin=-1, xmax=4.470, ymin=-1.2, ymax=3.0,
  xlabel={$s=\log t$ (logarithmic time)}, ylabel={$t^{-1/\alpha}(x+L_{t})$},
  xtick={0.535,1.531,2.262,2.896,3.485,4.050}, xticklabels={$\log u_{1}$,$\log u_{2}$,$\log u_{3}$,$\log u_{4}$,$\log u_{5}$,$\log u_{6}$},
  ytick={-2,-1,0,1,2,3},
  tick label style={font=\scriptsize}, label style={font=\small},
  every axis plot/.append style={line width=0.6pt},
  axis on top]
\fill[black!7] (axis cs:-1,-1.2) rectangle (axis cs:4.470,0);
\draw[black!50,densely dashed] (axis cs:0.535,\pgfkeysvalueof{/pgfplots/ymin}) -- (axis cs:0.535,\pgfkeysvalueof{/pgfplots/ymax});
\draw[black!50,densely dashed] (axis cs:1.531,\pgfkeysvalueof{/pgfplots/ymin}) -- (axis cs:1.531,\pgfkeysvalueof{/pgfplots/ymax});
\draw[black!50,densely dashed] (axis cs:2.262,\pgfkeysvalueof{/pgfplots/ymin}) -- (axis cs:2.262,\pgfkeysvalueof{/pgfplots/ymax});
\draw[black!50,densely dashed] (axis cs:2.896,\pgfkeysvalueof{/pgfplots/ymin}) -- (axis cs:2.896,\pgfkeysvalueof{/pgfplots/ymax});
\draw[black!50,densely dashed] (axis cs:3.485,\pgfkeysvalueof{/pgfplots/ymin}) -- (axis cs:3.485,\pgfkeysvalueof{/pgfplots/ymax});
\draw[black!50,densely dashed] (axis cs:4.050,\pgfkeysvalueof{/pgfplots/ymin}) -- (axis cs:4.050,\pgfkeysvalueof{/pgfplots/ymax});
\addplot[no marks,black] coordinates {(-1.000,1.797) (-0.990,1.732) (-0.980,1.729) (-0.970,1.683) (-0.960,1.737) (-0.950,1.616) (-0.940,1.625) (-0.930,1.579) (-0.920,1.603) (-0.910,1.547) (-0.900,1.639) (-0.890,1.635) (-0.880,1.596) (-0.870,1.827) (-0.860,1.811) (-0.850,1.734) (-0.840,1.613) (-0.830,1.678) (-0.820,1.653) (-0.810,1.670) (-0.800,1.719) (-0.790,1.697) (-0.780,1.705) (-0.770,1.708) (-0.760,1.716) (-0.750,1.695) (-0.740,1.665) (-0.730,1.528) (-0.720,1.569) (-0.710,1.630) (-0.700,2.100) (-0.690,1.371) (-0.680,1.341) (-0.670,1.276) (-0.660,1.240) (-0.650,1.174) (-0.640,1.275) (-0.630,1.240) (-0.620,1.219) (-0.610,1.248) (-0.600,1.339) (-0.590,1.342) (-0.580,1.352) (-0.570,1.306) (-0.560,1.256) (-0.550,1.231) (-0.540,1.132) (-0.530,1.031) (-0.520,0.916) (-0.510,0.828) (-0.500,0.951) (-0.490,0.881) (-0.480,0.948) (-0.470,0.936) (-0.460,0.847) (-0.450,0.893) (-0.440,0.862) (-0.430,0.934) (-0.420,1.036) (-0.410,1.094) (-0.400,1.064) (-0.390,1.114) (-0.380,1.133) (-0.370,1.079) (-0.360,1.035) (-0.350,0.999) (-0.340,1.045) (-0.330,1.068) (-0.320,1.050) (-0.310,1.294) (-0.300,1.357) (-0.290,1.379) (-0.280,1.356) (-0.270,1.322) (-0.260,1.360) (-0.250,1.407) (-0.240,1.315) (-0.230,1.345) (-0.220,1.248) (-0.210,1.253) (-0.200,1.291) (-0.190,1.220) (-0.180,1.150) (-0.170,1.193) (-0.160,1.296) (-0.150,1.282) (-0.140,1.318) (-0.130,0.894) (-0.120,0.884) (-0.110,0.941) (-0.100,0.856) (-0.090,0.898) (-0.080,0.993) (-0.070,0.912) (-0.060,0.930) (-0.050,0.899) (-0.040,0.978) (-0.030,0.982) (-0.020,0.920) (-0.010,0.933) (0.000,0.917) (0.010,0.956) (0.020,0.926) (0.030,0.893) (0.040,1.045) (0.050,0.959) (0.060,1.010) (0.070,0.921) (0.080,0.916) (0.090,0.881) (0.100,0.902) (0.110,0.923) (0.120,0.892) (0.130,0.946) (0.140,0.989) (0.150,0.929) (0.160,0.839) (0.170,1.113) (0.180,1.045) (0.190,0.974) (0.200,0.890) (0.210,0.946) (0.220,0.919) (0.230,0.902) (0.240,0.843) (0.250,0.818) (0.260,1.114) (0.270,0.953) (0.280,0.951) (0.290,0.777) (0.300,0.818) (0.310,0.803) (0.320,0.851) (0.330,0.853) (0.340,0.924) (0.350,0.939) (0.360,1.022) (0.370,1.105) (0.380,1.290) (0.390,1.338) (0.400,1.295) (0.410,1.229) (0.420,1.092) (0.430,1.043) (0.440,1.117) (0.450,1.168) (0.460,1.129) (0.470,1.092) (0.480,1.077) (0.490,1.036) (0.500,1.046) (0.510,0.949) (0.520,0.993) (0.530,1.178) (0.540,1.171) (0.550,1.039) (0.560,1.048) (0.570,1.092) (0.580,1.154) (0.590,1.248) (0.600,1.275) (0.610,1.386) (0.620,1.432) (0.630,1.378) (0.640,1.437) (0.650,1.439) (0.660,1.401) (0.670,1.308) (0.680,1.365) (0.690,1.322) (0.700,1.331) (0.710,1.317) (0.720,1.373) (0.730,1.536) (0.740,1.712) (0.750,1.759) (0.760,1.704) (0.770,1.649) (0.780,1.581) (0.790,1.526) (0.800,1.586) (0.810,1.635) (0.820,1.685) (0.830,1.759) (0.840,1.721) (0.850,1.622) (0.860,1.684) (0.870,1.744) (0.880,1.694) (0.890,1.735) (0.900,1.671) (0.910,1.564) (0.920,1.573) (0.930,1.545) (0.940,1.443) (0.950,1.446) (0.960,1.405) (0.970,0.915) (0.980,0.891) (0.990,0.839) (1.000,0.900) (1.010,0.852) (1.020,0.823) (1.030,0.806) (1.040,0.805) (1.050,0.720) (1.060,0.779) (1.070,0.782) (1.080,0.814) (1.090,0.679) (1.100,0.609) (1.110,0.550) (1.120,0.484) (1.130,0.486) (1.140,0.507) (1.150,0.604) (1.160,0.633) (1.170,0.550) (1.180,0.302) (1.190,0.264) (1.200,0.261) (1.210,0.318) (1.220,0.343) (1.230,0.340) (1.240,0.540) (1.250,0.589) (1.260,0.553) (1.270,0.498) (1.280,0.481) (1.290,0.444) (1.300,0.481) (1.310,0.979) (1.320,0.999) (1.330,1.078) (1.340,1.003) (1.350,0.956) (1.360,1.015) (1.370,1.104) (1.380,1.112) (1.390,0.971) (1.400,0.877) (1.410,0.830) (1.420,0.876) (1.430,0.854) (1.440,0.803) (1.450,0.830) (1.460,0.845) (1.470,0.664) (1.480,0.699) (1.490,0.726) (1.500,0.762) (1.510,1.201) (1.520,1.157) (1.530,1.315) (1.540,1.304) (1.550,1.289) (1.560,1.290) (1.570,1.223) (1.580,0.969) (1.590,1.020) (1.600,1.057) (1.610,1.005) (1.620,0.965) (1.630,1.061) (1.640,0.992) (1.650,0.920) (1.660,0.856) (1.670,0.861) (1.680,0.907) (1.690,0.789) (1.700,0.722) (1.710,0.698) (1.720,0.760) (1.730,0.738) (1.740,0.791) (1.750,0.772) (1.760,0.803) (1.770,0.717) (1.780,0.712) (1.790,0.728) (1.800,0.773) (1.810,0.836) (1.820,0.775) (1.830,0.807) (1.840,0.761) (1.850,0.720) (1.860,0.792) (1.870,0.721) (1.880,0.758) (1.890,0.714) (1.900,0.717) (1.910,0.743) (1.920,0.821) (1.930,0.849) (1.940,0.762) (1.950,0.751) (1.960,0.695) (1.970,0.663) (1.980,0.667) (1.990,0.740) (2.000,0.778) (2.010,0.589) (2.020,0.669) (2.030,0.873) (2.040,0.864) (2.050,0.862) (2.060,0.754) (2.070,0.695) (2.080,0.742) (2.090,0.678) (2.100,0.627) (2.110,0.560) (2.120,0.589) (2.130,0.581) (2.140,0.542) (2.150,0.496) (2.160,0.457) (2.170,0.370) (2.180,0.458) (2.190,0.349) (2.200,0.387) (2.210,0.425) (2.220,0.428) (2.230,0.382) (2.240,0.337) (2.250,0.325) (2.260,0.223) (2.270,0.229) (2.280,0.228) (2.290,0.243) (2.300,0.179) (2.310,0.129) (2.320,0.091) (2.330,0.160) (2.340,0.116) (2.350,0.152) (2.360,-0.052) (2.370,-0.067) (2.380,-0.091) (2.390,-0.062) (2.400,-0.099) (2.410,-0.091) (2.420,-0.102) (2.430,-0.064) (2.440,-0.081) (2.450,-0.012) (2.460,-0.079) (2.470,0.832) (2.480,0.817) (2.490,0.794) (2.500,0.791) (2.510,0.740) (2.520,0.732) (2.530,0.755) (2.540,0.649) (2.550,0.600) (2.560,0.606) (2.570,0.652) (2.580,0.645) (2.590,0.710) (2.600,0.833) (2.610,0.861) (2.620,0.913) (2.630,1.100) (2.640,0.965) (2.650,0.978) (2.660,0.972) (2.670,1.080) (2.680,0.948) (2.690,0.909) (2.700,0.952) (2.710,1.037) (2.720,0.662) (2.730,0.686) (2.740,0.821) (2.750,0.812) (2.760,0.785) (2.770,0.760) (2.780,1.045) (2.790,1.058) (2.800,1.037) (2.810,1.085) (2.820,1.071) (2.830,1.066) (2.840,1.103) (2.850,1.086) (2.860,1.069) (2.870,1.037) (2.880,1.030) (2.890,0.993) (2.900,0.983) (2.910,0.964) (2.920,0.868) (2.930,0.873) (2.940,0.900) (2.950,0.967) (2.960,0.987) (2.970,0.978) (2.980,1.053) (2.990,0.933) (3.000,1.000) (3.010,0.940) (3.020,0.998) (3.030,1.070) (3.040,1.034) (3.050,1.007) (3.060,0.389) (3.070,0.232) (3.080,0.161) (3.090,0.201) (3.100,0.245) (3.110,0.260) (3.120,0.377) (3.130,0.422) (3.140,0.494) (3.150,0.414) (3.160,0.491) (3.170,0.510) (3.180,0.538) (3.190,0.627) (3.200,0.692) (3.210,0.601) (3.220,0.615) (3.230,0.596) (3.240,0.634) (3.250,0.605) (3.260,0.697) (3.270,0.636) (3.280,0.698) (3.290,0.674) (3.300,0.601) (3.310,0.571) (3.320,0.511) (3.330,0.620) (3.340,0.554) (3.350,0.530) (3.360,0.443) (3.370,0.337) (3.380,0.241) (3.390,0.000) (3.400,0.002) (3.410,-0.089) (3.420,-0.090) (3.430,-0.058) (3.440,-0.096) (3.450,-0.110) (3.460,-0.126) (3.470,-0.157) (3.480,-0.273) (3.490,-0.397) (3.500,-0.484) (3.510,-0.432) (3.520,-0.407) (3.530,-0.422) (3.540,-0.473) (3.550,-0.356) (3.560,-0.366) (3.570,-0.406) (3.580,-0.465) (3.590,-0.466) (3.600,-0.671) (3.610,-0.684) (3.620,-0.743) (3.630,-0.690) (3.640,-0.697) (3.650,-0.677) (3.660,-0.700) (3.670,-0.640) (3.680,-0.629) (3.690,-0.573) (3.700,-0.433) (3.710,-0.240) (3.720,-0.245) (3.730,-0.221) (3.740,-0.242) (3.750,-0.150) (3.760,-0.015) (3.770,-0.043) (3.780,0.076) (3.790,0.109) (3.800,0.126) (3.810,0.055) (3.820,-0.174) (3.830,-0.236) (3.840,-0.071) (3.850,-0.104) (3.860,-0.149) (3.870,-0.336) (3.880,-0.352) (3.890,-0.371) (3.900,-0.328) (3.910,-0.364) (3.920,-0.346) (3.930,-0.371) (3.940,-0.273) (3.950,-0.274) (3.960,-0.213) (3.970,-0.243) (3.980,-0.249) (3.990,-0.314) (4.000,-0.357) (4.010,-0.449) (4.020,-0.444) (4.030,-0.365) (4.040,-0.341) (4.050,-0.360) (4.060,-0.215) (4.070,-0.033) (4.080,-0.078) (4.090,-0.030) (4.100,0.034) (4.110,-0.044) (4.120,0.060) (4.130,0.043) (4.140,-0.027) (4.150,-0.029) (4.160,-0.217) (4.170,-0.190) (4.180,-0.214) (4.190,-0.187) (4.200,-0.386) (4.210,-0.388) (4.220,-0.291) (4.230,-0.230) (4.240,-0.226) (4.250,-0.199) (4.260,-0.210) (4.270,-0.137) (4.280,-0.076) (4.290,-0.058) (4.300,0.085) (4.310,0.058) (4.320,0.105) (4.330,0.195) (4.340,0.235) (4.350,0.173) (4.360,0.177) (4.370,0.141) (4.380,0.206) (4.390,0.231) (4.400,0.184) (4.410,0.222) (4.420,0.321) (4.430,0.230) (4.440,0.306) (4.450,0.262) (4.460,0.268)};
\addplot[only marks,mark=*,mark size=1.9pt,black] coordinates {(0.535,1.230) (1.531,1.307) (2.262,0.222) (2.896,0.972)};
\addplot[only marks,mark=o,mark size=2.1pt,black,line width=0.8pt] coordinates {(3.485,-0.320) (4.050,-0.360)};
\draw[<->,>=stealth,black!70] (axis cs:3.485,2.65) -- (axis cs:4.050,2.65);
\node[font=\scriptsize,anchor=north,black!70] at (axis cs:3.767,2.57) {$\approx h$};
\node[font=\small,anchor=north west,inner sep=2pt] at (rel axis cs:0.02,0.98) {(b)};
\end{axis}
\end{tikzpicture}
\caption{The geometric-time representation on one path ($\alpha=3/2$, $a=0.7$, $\sigma=1$, $x=1$). (a) A symmetric stable path $x+L_{t}$ and the observation times $u_{1}, \dots, u_{6}$ of \eqref{eq:un}; by \eqref{eq:embedding} the signs of the marked values are those of $X_{1}, \dots, X_{6}$. The chain survives to $u_{4}$ and exits at $u_{5}$ (open marker), although the path already crossed below zero between $u_{3}$ and $u_{4}$: continuous survival fails where sampled survival does not. (b) The same path in logarithmic time after the rescaling $t^{-1/\alpha}$, the Lamperti transform of Section~\ref{sec:lamperti} applied to $x+L$. The observation times $\log{u_{n}}$ approach the uniform mesh $h=\alpha\log{(1/a)}$ of \eqref{eq:grid-ratio}, and the marked values equal $X_{n}$ up to the factor $\sigma^{-1}(1-a^{\alpha})^{1/\alpha}(1-a^{\alpha n})^{-1/\alpha}$.}
\label{fig:grid}
\end{figure}
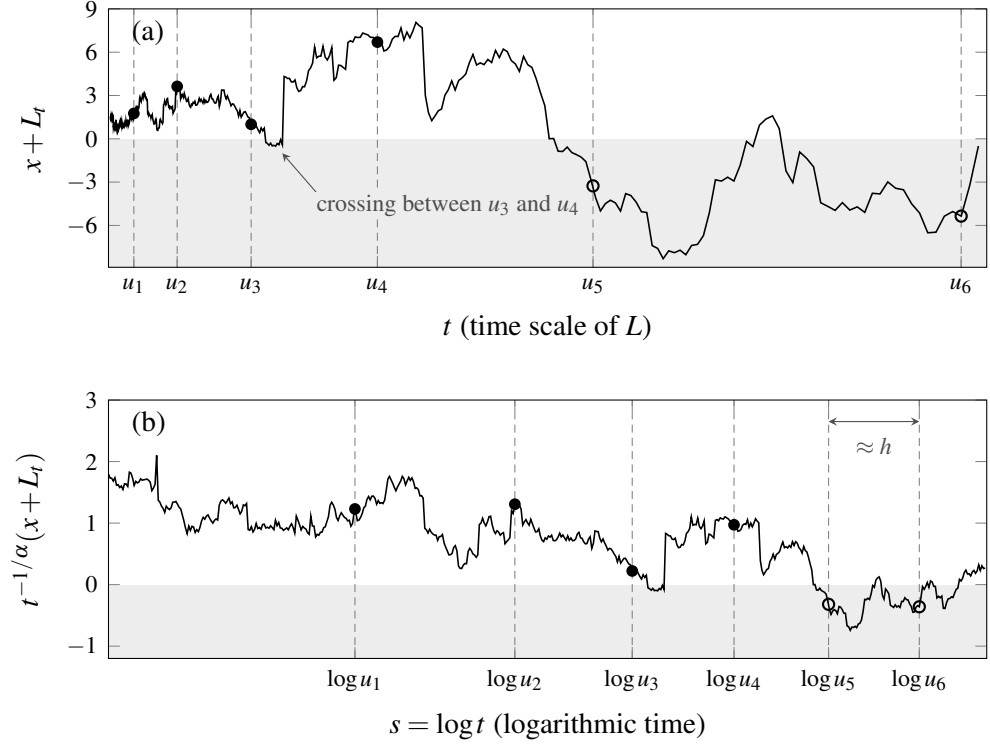


\section{Bounds for symmetric and asymmetric stable innovations}
\label{sec:bounds}

\subsection{Stable parameterization and a half-line exit estimate}

For the symmetric law in \eqref{eq:stable-cf}, set $\beta=0$ and $\rho=1/2$. To include asymmetric strictly stable laws, let $\alpha\neq1$ and replace \eqref{eq:stable-cf} by
\begin{equation}
\label{eq:stable-cf-asym}
\EE[e^{iq\xi_{1}}] =
\exp\bigl[-\sigma^{\alpha}|q|^{\alpha} 
(1-i\beta\operatorname{sgn}(q)\tan(\tfrac{1}{2}\pi\alpha))\bigr], 
\quad \beta \in [-1,1],
\end{equation}
where the positivity parameter is \cite{Zolotarev1986}
\begin{equation}
\label{eq:rho-def}
\rho \deq \PP(\xi_{1}>0) =
\frac{1}{2}+\frac{1}{\pi\alpha}\arctan\bigl(\beta\tan(\tfrac{1}{2}\pi\alpha)\bigr).
\end{equation}
We assume
\begin{equation}
\label{eq:rho-range}
\rho \in (0,1),
\end{equation}
which excludes the one-sided distributions. For $1 < \alpha < 2$, admissibility automatically restricts $\rho$ to $[1-1/\alpha,1/\alpha]$: here $\tfrac{1}{2}\pi\alpha \in (\pi/2,\pi)$, so the principal branch gives $\arctan(\pm\tan(\tfrac{1}{2}\pi\alpha)) = \pm(\tfrac{1}{2}\pi\alpha-\pi)$, and \eqref{eq:rho-def} returns $\rho=1-1/\alpha$ at $\beta=1$ and $\rho=1/\alpha$ at $\beta=-1$; reading $\arctan\circ\tan$ as the identity would instead return the values $\rho=1$ and $\rho=0$ proper to $0<\alpha<1$. At $\alpha=2$ only $\rho=1/2$ occurs. The restriction $\alpha\neq1$ is not vacuous: for $\alpha\neq1$ strict stability forces the shift to vanish, so \eqref{eq:stable-cf-asym} is the general strictly $\alpha$-stable characteristic function, whereas a strictly $1$-stable law is a Cauchy law with an arbitrary location parameter and is symmetric only when that parameter vanishes \cite[Properties~1.2.6 and~1.2.8, Corollary~1.2.10]{SamorodnitskyTaqqu1994}. We exclude it for the analytic reason given in Remark~\ref{rem:cauchy-gap}. We call the laws \eqref{eq:stable-cf-asym}--\eqref{eq:rho-range}, and the L\'{e}vy processes they generate, \emph{admissible}.

For the chain driven by \eqref{eq:stable-cf-asym}, append $\rho$ to the symmetric notation and write
\[
Q_{n}(x;a,\alpha,\sigma,\rho), \quad
\Lambda(a,\alpha,\rho), \quad
G_{\alpha,\rho}(a) = \frac{\Lambda(a,\alpha,\rho)}{\log(1/a)}.
\]
We suppress $\rho=1/2$ in the symmetric notation. The hypotheses of \cite[Theorem~1]{HKW2020} involve the innovations only through $\EE[\log{(1+|\xi_{1}|)}]<\infty$, $\EE[(\xi_{1}^{+})^{\delta}]<\infty$ for every $0<\delta<\alpha$, and ${\PP(\xi_{1}>0)\PP(\xi_{1}<0) > 0}$, all of which hold under \eqref{eq:rho-range}; the conclusion of Proposition~\ref{prop:hkw} therefore applies unchanged. Positive linear combinations of independent copies of $\xi_{1}$ remain strictly $\alpha$-stable with the same $\beta$ and $\rho$; in particular, the stationary scale is still \eqref{eq:stationary-scale}, and the stationary law has support all of $\RR$.

From here on $L$ denotes the standard strictly $\alpha$-stable L\'{e}vy process of the same type as the innovations, with $\EE[e^{iqL_{t}}] = e^{-t\psi(q)}$ and
\begin{equation}
\label{eq:psi-def}
\psi(q) = |q|^{\alpha}\bigl(1-i\beta\operatorname{sgn}(q)\tan(\tfrac{1}{2}\pi\alpha)\bigr),
\end{equation}
which reduces to \eqref{eq:levy-cf} when $\beta=0$. Since the increment of $L$ over $u_{n}-u_{n-1} = \sigma^{\alpha}a^{-\alpha n}$ has the same scale as $a^{-n}\xi_{n}$ for every $\beta$, Theorem~\ref{thm:embedding} and Corollary~\ref{cor:stationary-scale} hold verbatim for this $L$.

For $y > 0$, let
\[
\tau_{0}^{L}(y) \,\deq\, \inf\{t > 0\colon y+L_{t} \leq 0\}
\]
be the first exit time of $y+L$ from $(0,\infty)$, and write
\[
S_{t}^{-} = \sup_{0 \leq s \leq t}(-L_{s}).
\]
We abbreviate $\PP_{y}(\tau_{0}^{L} > T) \deq \PP(\tau_{0}^{L}(y) > T)$. This notation is also used below with a random starting point, where it denotes the function $y \mapsto \PP(\tau_{0}^{L}(y) > T)$ evaluated at that point.

\begin{lemma}[Stable exit asymptotic and global bound]
\label{lem:stable-exit}
Let $L$ be an admissible L\'{e}vy process with positivity parameter $\rho$. There is a constant $c_{\alpha,\rho} \in (0,\infty)$ such that, for each $x>0$,
\begin{equation}
\label{eq:exit-asymptotic}
\PP_{x}(\tau_{0}^{L} > T) \sim
c_{\alpha,\rho}x^{\alpha(1-\rho)}T^{-(1-\rho)}, \quad T \to \infty.
\end{equation}
There is also a constant $C_{\alpha,\rho} < \infty$ such that, for all $x,T>0$,
\begin{equation}
\label{eq:exit-uniform}
\PP_{x}(\tau_{0}^{L} > T) \leq C_{\alpha,\rho} \min\{1,x^{\alpha(1-\rho)}T^{-(1-\rho)}\}.
\end{equation}
In the symmetric case, $\rho=1/2$; we then write $c_{\alpha}=c_{\alpha,1/2}$ and $C_{\alpha} = C_{\alpha,1/2}$.
\end{lemma}

\begin{proof}
Let $F_{\alpha,\rho}(y)=\PP(S_{1}^{-}<y)$, $y>0$. The path $x+L$ remains in $(0,\infty)$ up to time $T$ exactly when $-L$ has not reached level $x$, so $\{\tau_{0}^{L}(x) > T\} = \{S_{T}^{-} < x\}$; and self-similarity gives $S_{T}^{-} \dis T^{1/\alpha}S_{1}^{-}$. Hence
\begin{equation}
\label{eq:exit-scaling}
\PP_{x}(\tau_{0}^{L} > T) = \PP(S_{T}^{-} < x) = 
F_{\alpha,\rho}(xT^{-1/\alpha}).
\end{equation}
For $0<\alpha<2$, the law of $L_{1}$ is continuous, so the process $-L$ has positivity parameter $\PP(-L_{1}>0)=\PP(L_{1}<0)=1-\rho$. The stable-supremum small-deviation estimate \cite{Bingham1973}, together with continuity of the supremum density \cite{DoneySavov2010}, gives
\begin{equation}
\label{eq:F-small}
F_{\alpha,\rho}(y) \sim c_{\alpha,\rho}y^{\alpha(1-\rho)}, \quad y \downarrow 0.
\end{equation}
When $\alpha=2$, one has $\rho=1/2$; the reflection principle gives $S_{1}^{-}\dis |L_{1}|$, and $L_{1} \sim N(0,2)$ has density $1/2\sqrt{\pi}$ at the origin, so $F_{2,1/2}(y) = \PP(|L_{1}|<y) \sim 2y \cdot 1/2\sqrt{\pi} = y/\sqrt{\pi}$, in agreement with the exponent $\alpha(1-\rho)=1$. Thus \eqref{eq:F-small} also holds in the Gaussian case.

Equations \eqref{eq:exit-scaling} and \eqref{eq:F-small} prove \eqref{eq:exit-asymptotic}. They also imply $F_{\alpha,\rho}(y) \leq C_{\alpha,\rho}y^{\alpha(1-\rho)}$ for all sufficiently small $y$. Enlarging $C_{\alpha,\rho}$ and using $F_{\alpha,\rho} \leq 1$ gives
\[
F_{\alpha,\rho}(y) \leq C_{\alpha,\rho}\min\{1,y^{\alpha(1-\rho)}\}, \quad y > 0,
\]
which, together with \eqref{eq:exit-scaling}, yields \eqref{eq:exit-uniform}.
\end{proof}

\begin{remark}[The asymmetric Cauchy case]
\label{rem:cauchy-gap}
For $\alpha=1$ the strictly stable laws are the Cauchy laws $\mathrm{Cauchy}(\mu,\sigma)$, $\mu \in \RR$, with
\[
\rho = \frac{1}{2}+\frac{1}{\pi}\arctan{(\mu/\sigma)} \in (0,1),
\]
so no one-sided law occurs and every $\rho \in (0,1)$ is admissible. The structural ingredients used in this section all persist: the process is $1$-self-similar, $\PP(L_{t}>0)$ does not depend on $t$, and positive linear combinations preserve $\mu/\sigma$ and hence $\rho$. What is missing is the small-deviation estimate \eqref{eq:F-small}. Bingham \cite{Bingham1973} takes the location parameter to be zero and explicitly excludes the asymmetric Cauchy process, listing its first-passage behavior as an open problem. The same exclusion persists in the later fluctuation theory on which sharp supremum asymptotics rest: Kuznetsov and Pardo \cite{KuznetsovPardo2013} work throughout with $\alpha \in (0,1)\cup(1,2)$ and with the parameterization \eqref{eq:stable-cf-asym}--\eqref{eq:rho-def}. We therefore leave the case $\alpha=1$, $\mu \neq 0$ open. Granted \eqref{eq:F-small}, every result of this section extends to it, and the $\alpha=1$ slice of Figure~\ref{fig:bounds}(b) is then filled as drawn.
\end{remark}

\subsection{Discrete and continuous-time bounds}

\begin{proposition}[Continuum comparison]
\label{prop:continuum-bound}
For the admissible laws, $x>0$ and $0<a<1$,
\begin{equation}
\label{eq:continuum-bound-rho}
\Lambda(a,\alpha,\rho) \leq \alpha(1-\rho)\log(1/a).
\end{equation}
In the symmetric case this becomes
\begin{equation}
\label{eq:continuum-bound}
\Lambda(a,\alpha) \leq \frac{\alpha}{2}\log{(1/a)}.
\end{equation}
\end{proposition}

\begin{proof}
By Theorem~\ref{thm:embedding}, valid for the present $L$ as noted after \eqref{eq:psi-def}, continuous survival up to $u_{n}$ implies positivity at all sampling times, and
\[
Q_{n}(x;a,\alpha,\sigma,\rho) \geq \PP_{x}(\tau_{0}^{L} > u_{n}).
\]
Lemma~\ref{lem:stable-exit} gives, as $n \to \infty$,
\[
\PP_{x}(\tau_{0}^{L} > u_{n}) \sim c_{\alpha,\rho}x^{\alpha(1-\rho)}u_{n}^{-(1-\rho)},
\]
so that, by \eqref{eq:grid-log},
\[
\log{\PP_{x}}(\tau_{0}^{L} > u_{n})
= -(1-\rho)\log{u_{n}}+O(1) = -\alpha(1-\rho)n\log(1/a)+O(1).
\]
Apply $-n^{-1}\log$ and let $n \to \infty$. For the amplitude, \eqref{eq:un} gives $u_{n} \sim \sigma^{\alpha}a^{-\alpha n}/(1-a^{\alpha})$, hence
\[
u_{n}^{-(1-\rho)} \sim \sigma^{-\alpha(1-\rho)}a^{\alpha(1-\rho)n}(1-a^{\alpha})^{1-\rho},
\]
and \eqref{eq:exit-asymptotic} then gives
\begin{equation}
\label{eq:amplitude-lower-bound}
\liminf_{n \to \infty}a^{-\alpha(1-\rho)n} Q_{n}(x;a,\alpha,\sigma,\rho) \geq 
c_{\alpha,\rho} \Bigl(\frac{x}{\sigma}\Bigr)^{\alpha(1-\rho)} (1-a^{\alpha})^{1-\rho}.
\end{equation}
\end{proof}

\begin{proposition}[Association bound]
\label{prop:association}
For the admissible laws, $n \geq 1$, and every $x \geq 0$, with \eqref{eq:Qn} extended to the boundary start as in Section~\ref{sec:introduction},
\begin{equation}
\label{eq:association-rho}
Q_{n}(x;a,\alpha,\sigma,\rho) \geq \rho^{n}, \quad 
\Lambda(a,\alpha,\rho) \leq \log(1/\rho).
\end{equation}
In the symmetric case,
\begin{equation}
\label{eq:association}
Q_{n}(x;a,\alpha,\sigma) \geq 2^{-n}, \quad \Lambda(a,\alpha) \leq \log{2}.
\end{equation}
\end{proposition}

\begin{proof}
Each event $\{X_{k} > 0\}$ is increasing in the independent innovations $\xi_{1},\ldots,\xi_{k}$, because
\[
X_{k} = a^{k} x+\sum_{j=1}^{k}a^{k-j}\xi_{j}
\]
has nonnegative coefficients. Independent random variables are associated, and increasing functions preserve association \cite{EsaryProschanWalkup1967}. Therefore
\[
\PP\Bigl(\bigcap_{k=1}^{n}\{X_{k} > 0\}\Bigr) \geq 
\prod_{k=1}^{n}\PP(X_{k} > 0).
\]
The random part of $X_{k}$ is strictly stable with positivity parameter $\rho$, while $a^{k}x \geq 0$. Thus every factor is at least $\rho$, proving \eqref{eq:association-rho}; the symmetric statement follows from $\rho=1/2$.
\end{proof}

\begin{proposition}[Continuous stable OU persistence]
\label{prop:continuous-ou}
Let $Y$ solve the equation in \eqref{eq:ou-sde}, now with an admissible driving process of positivity parameter $\rho$, and let $\theta,\eta,x>0$. Then
\begin{equation}
\label{eq:ou-full-asymptotic}
\PP_{x}(Y_{s}>0,~ 0 \leq s \leq t) \sim c_{\alpha,\rho}x^{\alpha(1-\rho)}
\Bigl(\frac{\alpha\theta}{\eta^{\alpha}}\Bigr)^{1-\rho} e^{-\alpha(1-\rho)\theta t}, \quad t \to \infty.
\end{equation}
Its continuous-time persistence rate is $\alpha(1-\rho)\theta$; this is $\alpha\theta/2$ in the symmetric case.
\end{proposition}

\begin{proof}
The mild solution is
\[
Y_{t} = e^{-\theta t}\Bigl(x+\eta\int_{0}^{t} e^{\theta s} \dd L_{s}\Bigr).
\]
The process
\[
Z_{t} = \eta\int_{0}^{t} e^{\theta s} \dd L_{s}
\]
has independent increments, and its increment over $[r,t]$ has characteristic exponent
\[
\psi(q)\,\eta^{\alpha} \int_{r}^{t} e^{\alpha\theta s} \dd s,
\]
with $\psi$ as in \eqref{eq:psi-def}; the skewness factor is unchanged because the integrand $e^{\theta s}$ is positive. Matching these independent increments with those of $L$, as in the proof of Theorem~\ref{thm:embedding}, gives the process identity
\[
(Z_{t})_{t \geq 0} \dis (L_{\tau(t)})_{t \geq 0}, \quad
\tau(t) = \frac{\eta^{\alpha}(e^{\alpha\theta t}-1)}{\alpha\theta}.
\]
Since the factor $e^{-\theta t}$ is positive,
\[
\PP_{x}(Y_{s}>0,~ 0 \leq s \leq t) = \PP_{x}(\tau_{0}^{L} > \tau(t)).
\]
Now apply \eqref{eq:exit-asymptotic} and
\[
\tau(t)^{-(1-\rho)} \sim \Bigl(\frac{\alpha\theta}{\eta^{\alpha}}\Bigr)^{1-\rho} e^{-\alpha(1-\rho)\theta t}.
\]
\end{proof}

\subsection{The regularly-varying-tail conjecture}

Suppose that the right tail of the innovations has the form
\[
\PP(\xi_{1} > u) = u^{-r}\ell(u),
\]
where $\ell$ is slowly varying. The innovations are not assumed stable in this subsection, so we write $\lambda_{a}$ for the persistence rate of the resulting chain, in the notation of \cite{HKW2020}. Proposition~19 and Remark~20 there give $\lambda_{a} \leq r\log{(1/a)}$ and equality when $\ell(x) = O((\log{x})^{-2r-2})$. Remark~20 conjectures equality under regular variation alone.

\begin{corollary}[Counterexample to the stable specialization]
\label{cor:hkw-counterexample}
Let $0 < \alpha < 2$. Symmetric $\alpha$-stable innovations disprove the identity conjectured in \cite[Remark~20]{HKW2020}. Their right tail is regularly varying with index $-\alpha$, so $r=\alpha$ in the notation of \cite[Eq.~(50)]{HKW2020}, whereas
\begin{equation}
\label{eq:hkw-gap}
\Lambda(a,\alpha) \leq \frac{\alpha}{2}\log{(1/a)} < \alpha\log{(1/a)}, 
\quad 0 < a < 1.
\end{equation}
\end{corollary}

\begin{proof}
A symmetric $\alpha$-stable law with $0 < \alpha < 2$ has
\[
\PP(\xi_{1} > u) = u^{-\alpha}\ell(u), \quad \ell(u) \longrightarrow c_{\alpha,\sigma}\in (0,\infty).
\]
Thus the conjectured value is $\alpha\log{(1/a)}$. Proposition~\ref{prop:continuum-bound} gives \eqref{eq:hkw-gap}.
\end{proof}

\begin{remark}[Scope of the counterexample]
The stable laws in Corollary~\ref{cor:hkw-counterexample} satisfy the regular-variation assumption in \cite[Remark~20]{HKW2020}, but their slowly varying factor converges to a positive constant and does not satisfy the additional decay hypothesis used for the proved equality in that paper. The corollary contradicts the conjectural extension, not the theorem under the stronger hypothesis. Heuristically, the value $\alpha\log{(1/a)}$ corresponds to a big-jump mechanism. An innovation on the scale $a^{-n}$ has probability of order $a^{\alpha n}$, and a single such jump lifts the chain to a height from which the geometric contraction alone keeps it positive for $O(n)$ further steps. At the fixed boundary this mechanism is dominated by the continuous-survival contribution $\gtrsim a^{\alpha n/2}$ from Proposition~\ref{prop:continuum-bound}. The factor $1/2$ comes from the symmetric stable half-line persistence exponent. Proposition~19 of \cite{HKW2020} instead concerns high-threshold first passage, where the big-jump mechanism is relevant. The Gaussian tail is not regularly varying.

For $0<\alpha<2$, $\alpha\neq1$, and the asymmetric parameterization \eqref{eq:stable-cf-asym}, the same counterexample applies whenever $\beta>-1$: the right tail is regularly varying with index $-\alpha$, while \eqref{eq:continuum-bound-rho} improves the conjectured exponent by the factor $1-\rho<1$. When $\beta=-1$ and $1 < \alpha < 2$, the innovations have no positive jumps and finite exponential moments of every positive order,
\[
\EE[e^{s\xi_{1}}] = \exp\bigl\{-\sigma^{\alpha}s^{\alpha}/\cos(\tfrac{1}{2}\pi\alpha)\bigr\}, \quad s \geq 0,
\]
the exponent being positive because $\cos(\tfrac{1}{2}\pi\alpha) < 0$ for $1 < \alpha < 2$ \cite[Proposition~1.2.12]{SamorodnitskyTaqqu1994}. The right tail is therefore not regularly varying but decays in Weibull fashion \cite[Eq.~(1.2.11)]{SamorodnitskyTaqqu1994},
\[
\log{\PP}(\xi_{1} > u) \sim -k_{\alpha,\sigma}u^{\alpha/(\alpha-1)}, 
\quad u \to \infty,
\]
where, writing $\widehat{\sigma}_{\alpha} = \sigma\bigl(\cos\tfrac{\pi}{2}(2-\alpha)\bigr)^{-1/\alpha}$,
\[
k_{\alpha,\sigma} =
(\alpha-1)\bigl(\alpha\widehat{\sigma}_{\alpha}\bigr)^{-\alpha/(\alpha-1)}.
\]
The exponent $\alpha/(\alpha-1)$ is the conjugate exponent of $\alpha$, as the Legendre transform of $s \mapsto \sigma^{\alpha}s^{\alpha}/|\cos(\tfrac{1}{2}\pi\alpha)|$ makes plain. So the big-jump mechanism is absent at the exponential persistence scale and the conjecture does not apply.
\end{remark}

\begin{theorem}[Asymmetric near-unit-root bounds]
\label{thm:main-asym}
Let $0 < \alpha \leq 2$, $\alpha\neq1$, and let $\rho$ be an admissible positivity parameter defined by \eqref{eq:rho-def}--\eqref{eq:rho-range}. Then, for every $0<a<1$,
\begin{equation}
\label{eq:main-bounds-asym}
0 < \Lambda(a,\alpha,\rho) \leq \min\bigl\{\log(1/\rho),\alpha(1-\rho)\log(1/a)\bigr\}.
\end{equation}
For every $m \in \NN$,
\begin{equation}
\label{eq:refinement-asym}
\Lambda(a,\alpha,\rho) \geq \frac{1}{m}\Lambda(a^{m},\alpha,\rho),
\end{equation}
and therefore
\begin{equation}
\label{eq:main-order-asym}
\Lambda(a,\alpha,\rho) \asymp \log(1/a)\quad(a \uparrow 1),
\quad
\lim_{a \uparrow 1}G_{\alpha,\rho}(a) = \sup_{0<a<1}G_{\alpha,\rho}(a) 
\eqd G_{\alpha,\rho}^{*} \in \bigl(0,\alpha(1-\rho)\bigr].
\end{equation}
\end{theorem}

\begin{proof}
Strict positivity follows from Proposition~\ref{prop:hkw}, and the two upper bounds are \eqref{eq:continuum-bound-rho} and \eqref{eq:association-rho}. The block sums in \eqref{eq:block-innovation} have positive coefficients and therefore retain the same $\beta$ and $\rho$. Thus the subsampling proof in Section~\ref{sec:refinement} gives \eqref{eq:refinement-asym}. The monotonicity coupling is distribution-free, so the anchored argument in Corollary~\ref{cor:order} applies unchanged except that $\alpha/2$ becomes $\alpha(1-\rho)$, and yields \eqref{eq:main-order-asym}.
\end{proof}

\begin{remark}[Admissible range of the asymmetric ceiling]
For $1 < \alpha < 2$, spectrally positive innovations have $\rho=1-1/\alpha$ and ceiling $\alpha(1-\rho)=1$, whereas spectrally negative innovations have $\rho=1/\alpha$ and ceiling $\alpha-1$. Thus the ceiling ranges over $[\alpha-1,1]$, takes the symmetric value $\alpha/2$ at $\rho=1/2$, and this range collapses to the single value $1$ as $\alpha \uparrow 2$. For $0<\alpha<1$ the two extremes exchange roles: $\beta=1$ gives $\rho=1$ and ceiling $0$, while $\beta=-1$ gives $\rho=0$ and ceiling $\alpha$, both excluded by \eqref{eq:rho-range}, so the ceiling ranges over the open interval $(0,\alpha)$. The excluded endpoints are consistent with the degenerate chains they describe: innovations that are almost surely positive give $Q_{n} \equiv 1$ and $\Lambda=0$.
\end{remark}


\section{Subsampling and the near-unit-root limit}
\label{sec:refinement}

\subsection{Stable closure under subsampling}

\begin{lemma}[Subsampling]
\label{lem:subsampling}
Fix $m \in \NN$ and put $Y_{j} = X_{jm}$. Then
\begin{equation}
\label{eq:subsampled-ar}
Y_{j} = a^{m}Y_{j-1}+\zeta_{j}, \quad Y_{0} = x,
\end{equation}
where
\begin{equation}
\label{eq:block-innovation}
\zeta_{j} = \sum_{i=1}^{m} a^{m-i}\xi_{(j-1)m+i}
\end{equation}
are independent $\alpha$-stable variables with scale
\begin{equation}
\label{eq:block-scale}
\sigma_{m} = \sigma\Bigl(\frac{1-a^{\alpha m}}{1-a^{\alpha}}\Bigr)^{1/\alpha},
\end{equation}
and with the same skewness $\beta$ and positivity parameter $\rho$ as the
original innovations. In the symmetric case, $\beta=0$ and $\rho=1/2$.
\end{lemma}

\begin{proof}
Iterating \eqref{eq:ar1} over a block of length $m$ gives
\[
X_{jm} = a^{m}X_{(j-1)m} + \sum_{i=1}^{m}a^{m-i}\xi_{(j-1)m+i},
\]
which is \eqref{eq:subsampled-ar}--\eqref{eq:block-innovation}. Distinct block sums use disjoint innovations. Stability and the geometric sum of Corollary~\ref{cor:stationary-scale},
\[
\sum_{i=1}^{m}a^{\alpha(m-i)} = \frac{1-a^{\alpha m}}{1-a^{\alpha}},
\]
give \eqref{eq:block-scale}; because every coefficient is nonnegative, the
skewness and positivity parameter are unchanged.
\end{proof}

\begin{theorem}[Subsampling inequality]
\label{thm:refinement}
For every $m \in \NN$,
\begin{equation}
\label{eq:refinement}
\Lambda(a,\alpha) \geq \frac{1}{m}\Lambda(a^{m},\alpha), \quad
G_{\alpha}(a) \geq G_{\alpha}(a^{m}).
\end{equation}
For the asymmetric laws of Section~\ref{sec:bounds}, the same statement holds
with $\Lambda(a,\alpha)$ and $G_{\alpha}(a)$ replaced by
$\Lambda(a,\alpha,\rho)$ and $G_{\alpha,\rho}(a)$.
\end{theorem}

\begin{proof}
Let $r = \lfloor n/m\rfloor$. Dropping all positivity constraints whose indices are not multiples of $m$ enlarges the event,
\[
\{X_{k}>0,~ 1 \leq k \leq n\} \subseteq \{X_{jm}>0,~ 1 \leq j \leq r\}.
\]
By Lemma~\ref{lem:subsampling}, the event on the right is the length-$r$ persistence event for an AR($1$) sequence with coefficient $a^{m}$, innovation scale $\sigma_{m}$, and starting point $x$. Scale invariance and the starting-point independence of the exponential rate (Proposition~\ref{prop:hkw}(i)) imply
\[
-\lim_{r \to \infty}\,\frac{1}{r}\log{\PP(X_{jm}>0,~ 1 \leq j \leq r)} =
\Lambda(a^{m},\alpha,\rho),
\]
in the general notation of Section~\ref{sec:bounds}. The inclusion gives $\PP(X_{k}>0,~ 1 \leq k \leq n) \leq \PP(X_{jm}>0,~ 1 \leq j \leq r)$, so
\[
-\frac{1}{n}\log{\PP(X_{k}>0,~ 1 \leq k \leq n)} \geq
\frac{r}{n}\biggl(-\frac{1}{r}\log{\PP(X_{jm}>0,~ 1 \leq j \leq r)}\biggr)
\longrightarrow \frac{1}{m}\Lambda(a^{m},\alpha,\rho),
\]
since $r/n \to 1/m$; this proves the asymmetric version, and the symmetric version is obtained by suppressing $\rho=1/2$. The normalized inequality follows from $\log(1/a^{m}) = m\log(1/a)$.
\end{proof}

For $\alpha=2$, the coarse-grid inclusion underlying Theorem~\ref{thm:refinement} appears in the proof of \cite[Theorem~2.7]{AurzadaBaumgarten2011} (Lemma~2.10 there). Stable closure of the block innovations extends the subsampling inequality \eqref{eq:refinement} to the full range $0 < \alpha \leq 2$.

\subsection{Monotonicity and consequences}

\begin{lemma}[Monotonicity]
\label{lem:monotonicity}
For fixed $x>0$ and $n$, the persistence probability is nondecreasing in $a \in (0,1)$, for both the symmetric and asymmetric laws considered here. Hence $a \mapsto \Lambda(a,\alpha,\rho)$ is nonincreasing; in the symmetric case this reads $a \mapsto \Lambda(a,\alpha)$.
\end{lemma}

\begin{proof}
Fix $0 < a_{1} < a_{2} < 1$, and couple the two chains using the same innovations and starting point. On the event that the $a_{1}$-chain is positive through time $n$, induction gives $X_{j}(a_{2}) \geq X_{j}(a_{1})$ for every $j \leq n$. Indeed,
\[
X_{j}(a_{2})-X_{j}(a_{1}) = 
a_{2}(X_{j-1}(a_{2})-X_{j-1}(a_{1})) + (a_{2}-a_{1})X_{j-1}(a_{1}),
\]
and the right-hand side is nonnegative whenever the induction hypothesis holds and $X_{j-1}(a_{1}) \geq 0$. Thus survival for $a_{1}$ implies survival for $a_{2}$. Passing to exponential rates gives the second assertion. This coupling comparison is also used in \cite[Section~2.2.1]{AurzadaMukherjeeZeitouni2021}.
\end{proof}

\begin{remark}[Where stability enters]
The near-unit lower bound has three ingredients. The monotonicity coupling of Lemma~\ref{lem:monotonicity} uses only that $(x,a) \mapsto ax+\xi$ is nondecreasing in $x$ and, for $x \geq 0$, in $a$. Along the survival event for the smaller-$a$ chain, the preceding states are nonnegative, so the finite-horizon comparison holds for arbitrary innovations. The subsampling inclusion $\{X_{k}>0,~ 1 \leq k \leq n\} \subseteq \{X_{jm}>0,~ 1 \leq j \leq r\}$ in Theorem~\ref{thm:refinement} is equally general: it merely drops constraints. Stability enters only in Lemma~\ref{lem:subsampling}, where closure under the block sums \eqref{eq:block-innovation} identifies the subsampled chain, up to scale, with the same family at coefficient $a^{m}$. The subsampling inequality \eqref{eq:refinement} therefore extends to any innovation family for which the block sums \eqref{eq:block-innovation} remain within the family up to a scale that does not affect the exponential rate.
\end{remark}

\begin{corollary}[Near-unit limit and global supremum]
\label{cor:order}
For each fixed $a_{0} \in (0,1)$, inequalities \eqref{eq:ceiling-lower} and \eqref{eq:quantitative-order} hold for all $a \in [a_{0},1)$. Moreover,
\begin{equation}
\label{eq:limit-sup}
\lim_{a \uparrow 1}G_{\alpha}(a) = \sup_{0<a<1}G_{\alpha}(a).
\end{equation}
\end{corollary}

\begin{proof}
Let
\[
A = \log{(1/a_{0})}, \quad d = \log{(1/a)}, \quad
M = \Bigl\lceil\frac{A}{d}\Bigr\rceil.
\]
For $a \in [a_{0},1)$, one has $d \leq A$; and $M = \lceil A/d\rceil$ gives $A \leq Md < A+d$, whence $a^{M} \leq a_{0}$. Theorem~\ref{thm:refinement} and Lemma~\ref{lem:monotonicity} therefore give
\[
\Lambda(a,\alpha) \geq
\frac{1}{M}\Lambda(a^{M},\alpha) \geq
\frac{\Lambda(a_{0},\alpha)}{M} \,=\, \frac{\Lambda(a_{0},\alpha)\,d}{Md} 
\,>\, \frac{\Lambda(a_{0},\alpha)\,d}{A+d},
\]
using $Md < A+d$ in the last step. This proves \eqref{eq:ceiling-lower} and the lower bound in \eqref{eq:quantitative-order}. The upper bound is Proposition~\ref{prop:continuum-bound}. Since $\Lambda(a_{0},\alpha) = AG_{\alpha}(a_{0})$ by \eqref{eq:G-def}, dividing the last display by $d$ gives
\[
G_{\alpha}(a) \geq \frac{A}{Md}G_{\alpha}(a_{0}) \geq \frac{A}{A+d}G_{\alpha}(a_{0}).
\]
Letting $a \uparrow 1$, and hence $d \downarrow 0$, shows that the liminf of $G_{\alpha}(a)$ is at least $G_{\alpha}(a_{0})$. Taking the supremum over $a_{0} \in (0,1)$ gives the lower bound in \eqref{eq:limit-sup}. The reverse bound follows from the definition of the supremum.
\end{proof}

\begin{remark}[No monotonicity claim for $G_{\alpha}$]
The identity \eqref{eq:limit-sup} does not assert that $a \mapsto G_{\alpha}(a)$ is monotone. The subsampling comparison gives $G_{\alpha}(a)\geq G_{\alpha}(a^{m})$ only along exact powers; the limit follows instead from the anchored lower bounds and the global-supremum argument above.
\end{remark}

\begin{proof}[Proof of Theorem~\ref{thm:main}]
The strict positivity in \eqref{eq:main-bounds} is part of Proposition~\ref{prop:hkw}(i). The upper bounds follow from Propositions~\ref{prop:continuum-bound} and \ref{prop:association}. Corollary~\ref{cor:order} gives \eqref{eq:ceiling-lower}, \eqref{eq:quantitative-order}, \eqref{eq:near-unit-order}, and \eqref{eq:limit-sup-intro}. Finally, Proposition~\ref{prop:continuum-bound} implies $G_{\alpha}^{*} \leq \alpha/2$, while positivity at any fixed $a_{0}$ gives $G_{\alpha}^{*} > 0$.
\end{proof}


\section{The stationary Lamperti representation}
\label{sec:lamperti}

\subsection{Continuous persistence in logarithmic time}

Define the Lamperti transform of stable L\'{e}vy motion
\cite{Lamperti1962} by
\begin{equation}
\label{eq:R-def}
R_{t} = e^{-t/\alpha}L_{e^{t}}, \quad t \in \RR.
\end{equation}
We write $\SaS(1)$ for the symmetric $\alpha$-stable law with characteristic function $e^{-|q|^{\alpha}}$.

\begin{lemma}[Stationarity and transitions]
\label{lem:R-transition}
The process $R$ is stationary and Markov, with marginal law $\SaS(1)$. For $t \in \RR$ and $h > 0$,
\begin{equation}
\label{eq:R-transition}
R_{t+h} = e^{-h/\alpha}R_{t} + (1-e^{-h})^{1/\alpha}Z_{t,h},
\end{equation}
where $Z_{t,h} \sim \SaS(1)$ is independent of $\sigma(R_{s}\colon s \leq t)$.
\end{lemma}

\begin{proof}
Self-similarity with $c = e^{t}$ gives $L_{e^{t}} \dis e^{t/\alpha}L_{1}$, hence $R_{t} = e^{-t/\alpha}L_{e^{t}} \dis L_{1}$. More generally, for every $s \in \RR$ and every finite collection of times $t_{1},\ldots,t_{k}$, the self-similarity relation with $c = e^{s}$ gives 
\[
(L_{e^{s}e^{t_{i}}})_{1 \leq i \leq k} \,\dis\, (e^{s/\alpha}L_{e^{t_{i}}})_{1 \leq i \leq k},
\]
and multiplying the $i$-th coordinate on each side by $e^{-(t_{i}+s)/\alpha}$ yields
\[
\bigl(R_{t_{1}+s}, \ldots, R_{t_{k}+s}\bigr) \dis 
\bigl(R_{t_{1}}, \ldots, R_{t_{k}}\bigr),
\]
so $R$ is stationary. Next,
\[
R_{t+h} = e^{-h/\alpha}R_{t} + e^{-(t+h)/\alpha}\bigl(L_{e^{t+h}}-L_{e^{t}}\bigr).
\]
The increment in parentheses is independent of the past and has stable scale
\[
(e^{t+h}-e^{t})^{1/\alpha} = e^{t/\alpha}(e^{h}-1)^{1/\alpha}.
\]
After multiplication by $e^{-(t+h)/\alpha}$, its scale becomes $(1-e^{-h})^{1/\alpha}$. This proves \eqref{eq:R-transition} and the Markov property.
\end{proof}

\begin{lemma}[Continuous log-time survival]
\label{lem:R-continuous}
There is a constant $\kappa_{\alpha} \in (0,\infty)$ such that
\begin{equation}
\label{eq:R-continuous}
\PP(R_{t}>0,~ 0 \leq t \leq T) \sim \kappa_{\alpha} e^{-T/2},
\quad T \to \infty.
\end{equation}
In particular, the continuous persistence rate of $R$ is $1/2$.
\end{lemma}

\begin{proof}
Since $e^{-t/\alpha} > 0$, one has $R_{t} > 0$ if and only if $L_{e^{t}} > 0$, and $t$ ranges over $[0,T]$ exactly when $e^{t}$ ranges over $[1,e^{T}]$. The event in \eqref{eq:R-continuous} is therefore
\[
\{L_{s}>0,~ 1 \leq s \leq e^{T}\}.
\]
By the Markov property at time $1$,
\begin{equation}
\label{eq:random-start}
\PP(R_{t}>0,~ 0 \leq t \leq T) =
\EE\bigl[\1_{\{L_{1} > 0\}} \PP_{L_{1}}(\tau_{0}^{L} > e^{T}-1)\bigr].
\end{equation}
Let $v = e^{T}-1$. Lemma~\ref{lem:stable-exit} gives, for each $y > 0$,
\[
v^{1/2}\,\PP_{y}(\tau_{0}^{L} > v) \longrightarrow c_{\alpha} y^{\alpha/2},
\]
and the global bound \eqref{eq:exit-uniform} gives the domination
\[
v^{1/2}\,\PP_{y}(\tau_{0}^{L} > v) \leq C_{\alpha} y^{\alpha/2}.
\]
The moment $\EE[(L_{1}^{+})^{\alpha/2}]$ is finite; its order $\alpha/2$ lies below the stable index when $0 < \alpha < 2$, while $L_{1}$ is Gaussian when $\alpha=2$. Dominated convergence in \eqref{eq:random-start} yields
\[
\PP(R_{t}>0,~ 0 \leq t \leq T) \sim
c_{\alpha}\EE[(L_{1}^{+})^{\alpha/2}](e^{T}-1)^{-1/2}.
\]
This proves the claim with $\kappa_{\alpha} = c_{\alpha}\EE[(L_{1}^{+})^{\alpha/2}]$.
\end{proof}

\subsection{Discrete sampling}

For $h > 0$, put
\begin{equation}
\label{eq:pn-def}
p_{n}(h) = \PP(R_{jh}>0,~ 1 \leq j \leq n).
\end{equation}

\begin{lemma}[Existence and elementary bounds]
\label{lem:lambda-exists}
The limit
\begin{equation}
\label{eq:lambda-h}
\lambda_{\alpha}(h) \,\deq\, -\lim_{n \to \infty}\,\frac{1}{n} \log{p_{n}(h)}
\end{equation}
exists and satisfies
\begin{equation}
\label{eq:lambda-basic}
0 \leq \lambda_{\alpha}(h) \leq h/2.
\end{equation}
\end{lemma}

\begin{proof}
For fixed $N$, set
\[
\Delta_{1} = L_{e^{h}}, \quad
\Delta_{k} = L_{e^{kh}}-L_{e^{(k-1)h}}, \quad 2 \leq k \leq N.
\]
These variables are independent, and
\[
R_{jh} = e^{-jh/\alpha}\sum_{k=1}^{j}\Delta_{k},
\quad 1 \leq j \leq N.
\]
Thus $(R_{h},\ldots,R_{Nh})$ is a coordinatewise nondecreasing linear function of independent variables and is therefore associated. The events $A = \{R_{jh}>0,~ 1 \leq j \leq m\}$ and $B = \{R_{jh}>0,~ m+1 \leq j \leq m+n\}$ are both increasing in $(\Delta_{1},\ldots,\Delta_{m+n})$, so association gives $\PP(A \cap B) \geq \PP(A)\PP(B)$, while stationarity gives $\PP(B) = p_{n}(h)$. Hence
\begin{equation}
\label{eq:pn-supermultiplicativity}
p_{m+n}(h) \geq p_{m}(h)p_{n}(h), \quad m, n \geq 1.
\end{equation}
Thus $-\log{p_{n}(h)}$ is subadditive, and Fekete's lemma proves existence of the limit.

Continuous survival on $[0,Nh]$ implies survival at the sampled times, so $p_{N}(h) \geq \PP(R_{t}>0,~ 0 \leq t \leq Nh) \sim \kappa_{\alpha}e^{-Nh/2}$ by Lemma~\ref{lem:R-continuous}; applying $-N^{-1}\log$ gives $\lambda_{\alpha}(h) \leq h/2$.
\end{proof}

\begin{theorem}[Lamperti reduction]
\label{thm:lamperti}
For every $h > 0$, the sequence $(R_{jh})_{j \geq 0}$ is the stationary symmetric stable AR($1$) chain
\begin{equation}
\label{eq:R-ar}
R_{(j+1)h} = e^{-h/\alpha}R_{jh}+\zeta_{j},
\end{equation}
where the innovations are independent symmetric $\alpha$-stable variables with scale $(1-e^{-h})^{1/\alpha}$. Setting $h = \alpha\log{(1/a)}$ yields
\begin{equation}
\label{eq:Lambda-lambda}
\Lambda(a,\alpha) = \lambda_{\alpha}(h)
\end{equation}
and
\begin{equation}
\label{eq:G-lambda}
0 < \frac{G_{\alpha}^{*}}{\alpha} = \lim_{h \downarrow 0}\frac{\lambda_{\alpha}(h)}{h} = \sup_{h > 0}\frac{\lambda_{\alpha}(h)}{h} \leq \frac{1}{2}.
\end{equation}
\end{theorem}

\begin{proof}
Equation \eqref{eq:R-ar} is \eqref{eq:R-transition} at $t = jh$. The transition noises for different $j$ come from disjoint increments of $L$ and are therefore independent. Also $R_{0} = L_{1} \sim \SaS(1)$, the stationary law of the recursion.

When $h = \alpha\log{(1/a)}$, the coefficient in \eqref{eq:R-ar} is $e^{-h/\alpha} = e^{-\log{(1/a)}} = a$, so $(R_{jh})_{j \geq 0}$ is the stationary version of \eqref{eq:ar1} with innovation scale $(1-e^{-h})^{1/\alpha}$, a scale that does not affect exponential rates. By stationarity, every window of $n+1$ consecutive positivity constraints has the same probability. In particular,
\[
\PP_{\pi}(X_{0}>0,X_{1}>0,\ldots,X_{n}>0) = \PP(R_{jh}>0,~ 0 \leq j \leq n) = p_{n+1}(h).
\]
Proposition~\ref{prop:hkw}(ii) applied to this chain therefore yields
\[
\lambda_{\alpha}(h) = -\lim_{n \to \infty}\,\frac{1}{n}\log{p_{n+1}(h)} = \Lambda(a,\alpha).
\]
This proves \eqref{eq:Lambda-lambda} and, since $\Lambda(a,\alpha) > 0$, strengthens the lower bound in \eqref{eq:lambda-basic} to strict positivity. Since $a \mapsto h = \alpha\log{(1/a)}$ is a bijection from $(0,1)$ to $(0,\infty)$,
\[
G_{\alpha}(a) = \alpha\frac{\lambda_{\alpha}(h)}{h},
\]
and \eqref{eq:limit-sup} proves \eqref{eq:G-lambda}.
\end{proof}

\begin{remark}[Dense-sampling reduction]
By \eqref{eq:G-lambda}, $G_{\alpha}^{*}/\alpha$ is the dense-sampling limit of the persistence rate per unit log-time for $R$: continuous monitoring has rate $1/2$, while sampling at mesh $h$ has rate $\lambda_{\alpha}(h)/h$.
\end{remark}

\begin{corollary}[Dense sampling of stable Ornstein--Uhlenbeck processes]
\label{cor:ou-dense}
Let $Y$ solve \eqref{eq:ou-sde}, with $\theta,\eta,x > 0$. For $\Delta > 0$, define its sampled persistence rate per unit physical time by
\begin{equation}
\label{eq:Gamma-ou}
\Gamma_{\alpha,\theta}(\Delta) \,\deq\, -\lim_{n \to \infty}\,\frac{1}{n\Delta} \log{\PP_{x}(Y_{j\Delta}>0,~ 1 \leq j \leq n)}.
\end{equation}
Then
\begin{equation}
\label{eq:Gamma-ou-limit}
\lim_{\Delta \downarrow 0}\Gamma_{\alpha,\theta}(\Delta) = \theta G_{\alpha}^{*} \in (0,\alpha\theta/2].
\end{equation}
Moreover, for every $m \in \NN$,
\begin{equation}
\label{eq:Gamma-refinement}
\Gamma_{\alpha,\theta}(\Delta/m) \geq \Gamma_{\alpha,\theta}(\Delta).
\end{equation}
The upper endpoint $\alpha\theta/2$ is the continuous-time persistence rate in Proposition~\ref{prop:continuous-ou}.
\end{corollary}

\begin{proof}
The $\Delta$-skeleton of $Y$ is an AR($1$) sequence with coefficient $e^{-\theta\Delta}$. Proposition~\ref{prop:hkw} and scale independence therefore give
\begin{equation}
\label{eq:Gamma-G-identity}
\Gamma_{\alpha,\theta}(\Delta) = \frac{\Lambda(e^{-\theta\Delta},\alpha)}{\Delta} = \theta G_{\alpha}(e^{-\theta\Delta}).
\end{equation}
Equation \eqref{eq:Gamma-ou-limit} follows from \eqref{eq:limit-sup}. Since $e^{-\theta\Delta} = (e^{-\theta\Delta/m})^{m}$, applying \eqref{eq:refinement} with $a = e^{-\theta\Delta/m}$ gives $G_{\alpha}(e^{-\theta\Delta/m}) \geq G_{\alpha}(e^{-\theta\Delta})$, which by \eqref{eq:Gamma-G-identity} is \eqref{eq:Gamma-refinement}.
\end{proof}

\begin{corollary}[Gaussian dense-sampling limit]
\label{cor:gaussian}
For $\alpha=2$,
\begin{equation}
\label{eq:gaussian-limit}
\lim_{h \downarrow 0}\frac{\lambda_{2}(h)}{h} = \frac{1}{2}, \quad
\lim_{a \uparrow 1}\frac{\Lambda(a,2)}{\log{(1/a)}} = 1.
\end{equation}
\end{corollary}

\begin{proof}
For $\alpha=2$, sign events are invariant under scaling, so we may replace $R$ by $R/\sqrt{2}$, a centered stationary Gaussian process with unit variance and covariance $e^{-|t|/2}$. Indeed, $\EE[L_{u}L_{v}] = 2\min\{u,v\}$ when $\alpha=2$, so $\EE[R_{t}R_{t+s}] = e^{-(2t+s)/2} \cdot 2e^{t} = 2e^{-s/2}$ for $s \geq 0$. The spectral density of $R/\sqrt{2}$ is the Fourier transform of $e^{-|t|/2}$ and is therefore proportional to $(\omega^{2}+1/4)^{-1}$. In the notation of \cite{FeldheimFM2025}, $\mathcal{M}$ is the class of spectral measures whose density at the origin exists and is positive (possibly infinite), and $\mathcal{L}$ is the class of spectral measures $\mu$ with logarithmic moment $\int \max\{\log^{1+\beta}{|\omega|},1\}\dd\mu(\omega)$ finite for some $\beta > 0$. The density above is finite and positive at the origin, and its $O(|\omega|^{-2})$ tail makes this moment finite for every $\beta > 0$ (polynomial decay dominates any power of the logarithm, $\int_{1}^{\infty}\omega^{-2}\log^{1+\beta}{\!\omega}\dd\omega < \infty$); thus the spectral measure of $R/\sqrt{2}$ belongs to $\mathcal{L} \cap \mathcal{M}$. The spectral density also satisfies
\[
\sup_{\omega \in \RR}\,\frac{\omega^{2}}{\omega^{2}+1/4} < \infty,
\]
which is the polynomial-decay condition of \cite[Remark~6]{FeldheimFM2025}, with decay exponent $2$. The discretized persistence exponent of \cite{FeldheimFM2025} is normalized per unit time and, on the sampling grid $h\mathbb{Z}$, equals $\lambda_{2}(h)/h$ in the present notation. The extension of Theorem~5(II) stated in \cite[Remark~6]{FeldheimFM2025} therefore gives the first limit in \eqref{eq:gaussian-limit}. The second follows from $h = 2\log{(1/a)}$ and \eqref{eq:Lambda-lambda}.
\end{proof}

\begin{remark}[The non-Gaussian case]
For $0 < \alpha < 2$, $R$ has no Gaussian component, and the value of its dense-sampling limit remains open.
\end{remark}

\begin{remark}[Asymmetric Lamperti reduction]
\label{rem:asym-lamperti}
For the asymmetric laws of Section~\ref{sec:bounds}, let $\lambda_{\alpha,\rho}(h)$ be the persistence exponent of $R_{t}=e^{-t/\alpha}L_{e^{t}}$ sampled at mesh $h$. The arguments of Lemmas~\ref{lem:R-transition}--\ref{lem:lambda-exists} and of Theorem~\ref{thm:lamperti} carry over verbatim, with the symmetric law $\SaS(1)$ replaced throughout by the law of $L_{1}$, but two conclusions change. In Lemma~\ref{lem:R-continuous} the continuous log-time rate becomes $1-\rho$ in place of $1/2$, the dominating moment being $\EE[(L_{1}^{+})^{\alpha(1-\rho)}]<\infty$; the ceiling in \eqref{eq:lambda-basic} accordingly becomes $\lambda_{\alpha,\rho}(h) \leq (1-\rho)h$. Consequently,
\begin{equation}
\label{eq:G-lambda-asym}
0 < \frac{G_{\alpha,\rho}^{*}}{\alpha} 
= \lim_{h \downarrow 0}\frac{\lambda_{\alpha,\rho}(h)}{h}
= \sup_{h>0}\frac{\lambda_{\alpha,\rho}(h)}{h}
\leq 1-\rho.
\end{equation}
For the corresponding stable Ornstein--Uhlenbeck process, the sampled rate per unit physical time converges to $\theta G_{\alpha,\rho}^{*}$, while continuous monitoring has rate $\alpha(1-\rho)\theta$ by Proposition~\ref{prop:continuous-ou}.
\end{remark}

\section{The remaining dense-sampling problems}
\label{sec:remaining}

\begin{conjecture}[Sharp near-unit constant]
\label{conj:sharp}
For every $0 < \alpha < 2$,
\begin{equation}
\label{eq:sharp-conjecture}
\lim_{a \uparrow 1}\frac{\Lambda(a,\alpha)}{\log{(1/a)}} = \frac{\alpha}{2}.
\end{equation}
Equivalently,
\begin{equation}
\label{eq:dense-conjecture}
\lim_{h \downarrow 0}\frac{\lambda_{\alpha}(h)}{h} = \frac{1}{2}.
\end{equation}
\end{conjecture}

The limits exist by Theorems~\ref{thm:main} and \ref{thm:lamperti}. By Corollary~\ref{cor:ou-dense}, the conjecture states that dense sampling recovers the continuous-time persistence rate.

\begin{conjecture}[Sharp constant, skewed case]
\label{conj:sharp-asym}
Let $0<\alpha<2$, $\alpha\neq1$, and let $\rho\neq1/2$ be an admissible positivity parameter defined by \eqref{eq:rho-def}--\eqref{eq:rho-range}. Then
\begin{equation}
\label{eq:sharp-asym}
\lim_{a \uparrow 1} \frac{\Lambda(a,\alpha,\rho)}{\log(1/a)} = \alpha(1-\rho),
\end{equation}
or equivalently,
\begin{equation}
\label{eq:dense-conjecture-asym}
\lim_{h \downarrow 0}\frac{\lambda_{\alpha,\rho}(h)}{h} = 1-\rho.
\end{equation}
\end{conjecture}

Theorem~\ref{thm:main-asym} and \eqref{eq:G-lambda-asym} establish the existence of the limits in the skewed case but do not identify their values.


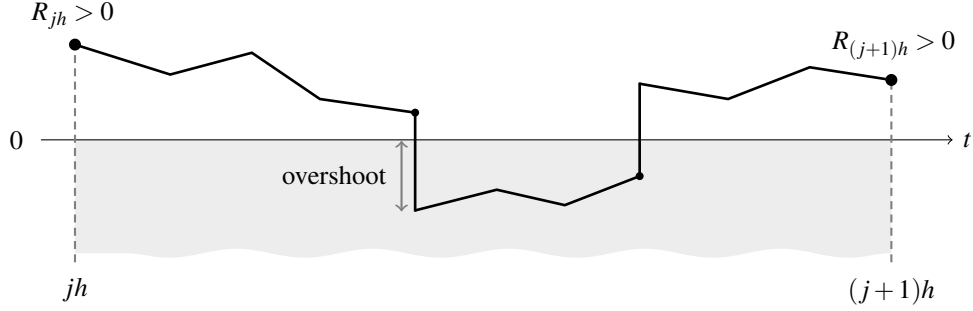
\begin{figure}[t]
\centering
\begin{tikzpicture}[x=1.80cm,y=1.20cm,font=\footnotesize]
\fill[black!7] (0,0) -- (6,0) -- (6,-1.25)
  decorate[decoration={snake,amplitude=0.6mm,segment length=18mm}] 
  { -- (0,-1.25) } -- cycle;
\draw[->] (-0.25,0) -- (6.45,0) node[right,font=\footnotesize] {$t$};
\node[font=\footnotesize,anchor=east] at (-0.3,0) {$0$};
\draw[black!50,densely dashed,thick] (0,-1.35)--(0,1.05);
\draw[black!50,densely dashed,thick] (6,-1.35)--(6,0.66);
\node[font=\footnotesize,anchor=north] at (0,-1.40) {$jh$};
\node[font=\footnotesize,anchor=north] at (6,-1.40) {$(j+1)h$};
\draw[line width=1.0pt] (0,1.05)--(0.7,0.72)--(1.30,0.96)--(1.8,0.45)--(2.5,0.30);
\draw[line width=1.0pt] (2.5,0.30)--(2.5,-0.79);
\fill (2.5,0.30) circle (1.5pt);
\draw[line width=1.0pt] (2.495,-0.78)--(3.1,-0.55)--(3.6,-0.72)--(4.15,-0.40);
\draw[line width=1.0pt] (4.15,-0.40) -- (4.15,0.63);
\fill (4.15,-0.40) circle (1.5pt);
\draw[line width=1.0pt] (4.145,0.62)--(4.8,0.45)--(5.4,0.80)--(6,0.66);
\fill (0,1.05) circle (2.25pt);
\fill (6,0.66) circle (2.25pt);
\node[font=\footnotesize,anchor=south west] at (-0.40,1.10) {$R_{jh}>0$};
\node[font=\footnotesize,anchor=south east] at (6.55,0.76) {$R_{(j+1)h}>0$};
\draw[<->,black!50,thick] (2.40,-0.01)--(2.40,-0.79);
\node[font=\footnotesize,anchor=west] at (1.44,-0.39) {overshoot};
\end{tikzpicture}
\caption{A path crosses below zero and returns above zero between successive observations. Both sampled values are positive, but continuous survival fails.}
\label{fig:submesh}
\end{figure}

The discrepancy between continuous and discrete persistence arises at the boundary: a sampled path may cross below zero and return above zero between successive observations. For a jump process, controlling such rescued sub-mesh crossings requires joint control of three quantities: the mass near zero under the conditioned law, the depth of the first overshoot below zero, and the probability of returning above zero during the remainder of the mesh interval. An estimate conditioned on a fixed overshoot depth is insufficient because the first overshoot can be arbitrarily small. Figure~\ref{fig:submesh} shows the configuration.

\begin{remark}[Remaining obstacle]
If the first downcrossing occurs at time $s$ with overshoot depth $z$, the return probability during the remaining interval is
\begin{equation}
r_{h-s}(z), \quad
r_{t}(z) \deq \PP(\overline L_{t} > z) = \PP(\overline L_{1} > zt^{-1/\alpha}),
\quad
\overline{L}_{t} \deq \sup_{0 \leq u \leq t}L_{u}.
\end{equation}
The kernel $r_{t}(z)$ is encoded by the bivariate renewal measure of the ascending ladder time--height process, whereas the downcrossing overshoot is described by the descending ladder process \cite{DoneyKyprianou2006}. The missing estimate must therefore integrate the return kernel $r_{t}(z)$ and the overshoot law uniformly against the near-boundary law conditioned on survival. No particular power of $h$ is asserted here. The same three quantities govern Conjecture~\ref{conj:sharp-asym}, with $1-\rho$ in place of $1/2$, and the two extremes of Figure~\ref{fig:bounds}(b) degenerate in opposite ways: for $1 < \alpha < 2$ and $\beta=1$ there are no negative jumps, first passage below zero occurs by creeping, and the overshoot vanishes identically, whereas for $\beta=-1$ there are no positive jumps, the overshoot is genuine, and a rescue must be achieved by a continuous climb. Neither degeneration by itself settles the corresponding dense-sampling limit.
\end{remark}

Further questions concern nonzero thresholds, subexponential corrections to $Q_{n}$, and the law of the chain conditioned on long survival. In the damped, Gaussian-driven Klein--Gordon lattice of \cite{Smii2026}, one could study persistence of discretely monitored local fields or spatial averages. Linearization about a dynamically stable equilibrium of the lattice gives a Gaussian comparison with the sampled Ornstein--Uhlenbeck process. Corresponding persistence questions could be posed for nonlinear dynamics or stable L\'{e}vy forcing.


\section*{Statements and Declarations}

\noindent\textbf{Author contributions.}
J.R.G.M. conceived the project and prepared the initial draft. Both J.R.G.M. and B.S. contributed to the mathematical development, discussion of the results, and revision of the manuscript.

\medskip

\noindent\textbf{Funding.}
J.R.G.M. acknowledges partial financial support from Funda\c{c}\~{a}o de Amparo \`{a} Pesquisa do Estado de S\~{a}o Paulo -- FAPESP, Brazil, through research grant no.~2020/04475-7. B.S. acknowledges the grant provided by the Interdisciplinary Research Center for Intelligent Secure Systems at King Fahd University of Petroleum and Minerals under project \#INSS2602.

\medskip

\noindent\textbf{Competing interests.}
The authors have no relevant financial or non-financial interests to disclose.

\medskip

\noindent\textbf{Data availability.}
Data sharing is not applicable to this article as no new data were created or analyzed in this study.



\begin{thebibliography}{23}

\bibitem{AlsmeyerBRS2023}
G.~Alsmeyer, A.~Bostan, K.~Raschel, and T.~Simon,
\emph{Persistence for a class of order-one autoregressive processes and Mallows--Riordan polynomials},
Advances in Applied Mathematics \textbf{150} (2023), 102555.
\href{https://doi.org/10.1016/j.aam.2023.102555}{doi:10.1016/j.aam.2023.102555}.

\bibitem{AurzadaBaumgarten2011}
F.~Aurzada and C.~Baumgarten,
\emph{Survival probabilities of weighted random walks},
ALEA Latin American Journal of Probability and Mathematical Statistics
\textbf{8} (2011), 235--258.
\href{https://alea.impa.br/articles/v8/08-13.pdf}{alea.impa.br/articles/v8/08-13}.

\bibitem{AurzadaMukherjeeZeitouni2021}
F.~Aurzada, S.~Mukherjee, and O.~Zeitouni,
\emph{Persistence exponents in Markov chains},
Annales de l'Institut Henri Poincar\'{e} (B) Probabilit\'{e}s et Statistiques
\textbf{57} (2021),
1411--1441.
\href{https://doi.org/10.1214/20-AIHP1114}{doi:10.1214/20-AIHP1114}.

\bibitem{AurzadaSimon2015}
F.~Aurzada and T.~Simon,
\emph{Persistence probabilities and exponents},
in \emph{L\'{e}vy Matters V}, Lecture Notes in Mathematics \textbf{2149}, Springer, Cham, 2015, pp.~183--224.
\href{https://doi.org/10.1007/978-3-319-23138-9_3}{doi:10.1007/978-3-319-23138-9\_3}.

\bibitem{Bingham1973}
N.~H.~Bingham,
\emph{Maxima of sums of random variables and suprema of stable processes},
Zeitschrift f\"ur Wahrscheinlichkeitstheorie und Verwandte Gebiete
\textbf{26} (1973), 273--296.
\href{https://doi.org/10.1007/BF00534892}{doi:10.1007/BF00534892}.

\bibitem{BrayMajumdarSchehr2013}
A.~J.~Bray, S.~N.~Majumdar, and G.~Schehr,
\emph{Persistence and first-passage properties in nonequilibrium systems},
Advances in Physics \textbf{62} (2013), 225--361.
\href{https://doi.org/10.1080/00018732.2013.803819}{doi:10.1080/00018732.2013.803819}.

\bibitem{DenisovHKW2022}
D.~Denisov, G.~Hinrichs, M.~Kolb, and V.~Wachtel,
\emph{Persistence of autoregressive sequences with logarithmic tails},
Electronic Journal of Probability \textbf{27} (2022), paper no.~154.
\href{https://doi.org/10.1214/22-EJP879}{doi:10.1214/22-EJP879}.

\bibitem{DoneyKyprianou2006}
R.~A.~Doney and A.~E.~Kyprianou,
\emph{Overshoots and undershoots of L\'{e}vy processes},
The Annals of Applied Probability \textbf{16} (2006), 91--106.
\href{https://doi.org/10.1214/105051605000000647}{doi:10.1214/105051605000000647}.

\bibitem{DoneySavov2010}
R.~A.~Doney and M.~S.~Savov,
\emph{The asymptotic behavior of densities related to the supremum of a stable process},
The Annals of Probability \textbf{38} (2010), 316--326.
\href{https://doi.org/10.1214/09-AOP479}{doi:10.1214/09-AOP479}.

\bibitem{DonnartSimon2026}
T.~Donnart and T.~Simon,
\emph{Persistence probabilities of autoregressive chains with continuous innovations},
preprint (2026), arXiv:2604.05670 [math.PR].
\href{https://doi.org/10.48550/arXiv.2604.05670}{doi:10.48550/arXiv.2604.05670}.

\bibitem{EsaryProschanWalkup1967}
J.~D.~Esary, F.~Proschan, and D.~W.~Walkup,
\emph{Association of random variables, with applications},
The Annals of Mathematical Statistics \textbf{38} (1967), 1466--1474.
\href{https://doi.org/10.1214/aoms/1177698701}{doi:10.1214/aoms/1177698701}.

\bibitem{FeldheimFM2025}
N.~D.~Feldheim, O.~N.~Feldheim, and S.~Mukherjee,
\emph{Persistence and ball exponents for Gaussian stationary processes},
Communications on Pure and Applied Mathematics \textbf{78} (2025), 1949--2000.
\href{https://doi.org/10.1002/cpa.22255}{doi:10.1002/cpa.22255}.

\bibitem{HKW2020}
G.~Hinrichs, M.~Kolb, and V.~Wachtel,
\emph{Persistence of one-dimensional AR($1$)-sequences},
Journal of Theoretical Probability \textbf{33} (2020), 65--102.
\href{https://doi.org/10.1007/s10959-018-0850-0}{doi:10.1007/s10959-018-0850-0}.

\bibitem{KuznetsovPardo2013}
A.~Kuznetsov and J.~C.~Pardo,
\emph{Fluctuations of stable processes and exponential functionals of hypergeometric L\'{e}vy processes},
Acta Applicandae Mathematicae \textbf{123} (2013), 113--139.
\href{https://doi.org/10.1007/s10440-012-9718-y}{doi:10.1007/s10440-012-9718-y}.

\bibitem{Lamperti1962}
J.~Lamperti,
\emph{Semi-stable stochastic processes},
Transactions of the American Mathematical Society \textbf{104} (1962), 62--78.
\href{https://doi.org/10.1090/S0002-9947-1962-0138128-7}{doi:10.1090/S0002-9947-1962-0138128-7}.

\bibitem{MajumdarBrayEhrhardt2001}
S.~N.~Majumdar, A.~J.~Bray, and G.~C.~M.~A.~Ehrhardt,
\emph{Persistence of a continuous stochastic process with discrete-time sampling},
Physical Review E \textbf{64} (2001), 015101(R).
\href{https://doi.org/10.1103/PhysRevE.64.015101}{doi:10.1103/PhysRevE.64.015101}.

\bibitem{Masuda2004}
H.~Masuda,
\emph{On multidimensional Ornstein--Uhlenbeck processes driven by a general L\'{e}vy process},
Bernoulli \textbf{10} (2004), 97--120.
\href{https://doi.org/10.3150/bj/1077544605}{doi:10.3150/bj/1077544605}.

\bibitem{Patie2008}
P.~Patie,
\emph{$q$-invariant functions for some generalizations of the Ornstein--Uhlenbeck semigroup},
ALEA Latin American Journal of Probability and Mathematical Statistics
\textbf{4} (2008), 31--43.
\href{https://alea.impa.br/articles/v4/04-02.pdf}{alea.impa.br/articles/v4/04-02}.

\bibitem{SamorodnitskyTaqqu1994}
G.~Samorodnitsky and M.~S.~Taqqu,
\emph{Stable Non-Gaussian Random Processes},
Chapman \& Hall, New York, 1994.

\bibitem{Smii2026}
B.~Smii,
\emph{A novel stochastic approach of thermalization and symmetry breaking},
preprint (2026), arXiv:2602.10563 [math-ph].
\href{https://doi.org/10.48550/arXiv.2602.10563}{doi:10.48550/arXiv.2602.10563}.

\bibitem{SparreAndersen1953}
E.~Sparre Andersen,
\emph{On the fluctuations of sums of random variables},
Mathematica Scandinavica \textbf{1} (1953), 263--285.
\href{https://doi.org/10.7146/math.scand.a-10385}{doi:10.7146/math.scand.a-10385}.

\bibitem{VysotskyWachtel2023}
V.~Vysotsky and V.~Wachtel,
\emph{Persistence of AR($1$) sequences with Rademacher innovations and linear mod~1 transforms},
Ergodic Theory and Dynamical Systems \textbf{46} (2026), 2361--2412.
\href{https://doi.org/10.1017/etds.2026.10292}{doi:10.1017/etds.2026.10292}.

\bibitem{Zolotarev1986}
V.~M.~Zolotarev,
\emph{One-Dimensional Stable Distributions},
Translations of Mathematical Monographs \textbf{65}, American Mathematical Society, Providence, 1986.
\end{thebibliography}
\end{document}